\documentclass[11pt]{amsart}
\usepackage[dvipsnames]{xcolor}
\usepackage[utf8]{inputenc}
\usepackage{comment,enumerate,graphicx}
\usepackage{enumitem,amsmath,amssymb}
\usepackage{amsthm}
\usepackage{docmute} 
\usepackage{thmtools}
\usepackage{url}
\usepackage{caption}
\usepackage{subcaption}
\usepackage[top=23 mm, bottom=23 mm, left=21 mm, right= 21 mm]{geometry}
\usepackage{relsize}
\usepackage[hidelinks]{hyperref}
\makeatletter
\newcommand{\labelinthm}[1]{%
   \label{temp#1}
   \protected@write \@auxout {}{\string \newlabel{#1}{{\emph{\ref{temp#1}}}{\thepage}{\emph{\ref{temp#1}}}{temp#1}{}} }%
}
\makeatother
\usepackage{adjustbox}
\usepackage{tikz}
\usetikzlibrary{math,calc,intersections, scopes}
\usetikzlibrary{positioning,arrows,shapes,decorations.markings,decorations.pathreplacing, decorations.pathmorphing,matrix,patterns}
\tikzstyle{vertex}=[circle,draw=black,fill=black,inner sep=0,minimum size=5pt,text=white,font=\footnotesize]
\tikzstyle{redvertex}=[circle,draw=red,fill=red,inner sep=0,minimum size=5pt,text=white,font=\footnotesize]
\definecolor{amber}{rgb}{1.0, 0.75, 0.0}
\definecolor{darkgreen}{rgb}{0.18, 0.7, 0.46}

\declaretheorem[name=Theorem,numberwithin=section]{theorem}

\newtheorem{lemma}[theorem]{\bf Lemma}

\newtheorem{claim}[theorem]{\bf Claim}
\newtheorem{proposition}[theorem]{\bf Proposition}

\newtheorem*{theorem*}{\bf Theorem}

\theoremstyle{definition}

\newtheorem{definition}[theorem]{\bf Definition}

\newenvironment{claimproof}{%
  \renewcommand{\qedsymbol}{$\boxdot$}%
  \begin{proof}%
}{%
  \end{proof}%
  \renewcommand{\qedsymbol}{$\square$}%
}

\newcommand\claimproofend{\renewcommand{\qedsymbol}{$\boxdot$}
\end{proof}
\renewcommand{\qedsymbol}{$\square$}
}
\def\eps{\varepsilon}

\def\cN{\mathcal{N}}

\def\bE{\mathbb{E}}
\def\bR{\mathbb{R}}

\def\Pb{\mathbb{P}}

\def\be{\boldsymbol{e}}
\def\bx{\boldsymbol{x}}
\def\by{\boldsymbol{y}}
\def\bz{\boldsymbol{z}}

\DeclareMathOperator{\red}{red}
\DeclareMathOperator{\blue}{blue}
\DeclareMathOperator{\spn}{span}

\title{\vspace{-0.9cm} Gaussian random graphs and Ramsey numbers}
\date{}

\author{Zach Hunter}
\author{Aleksa Milojevi\'c}
\author{Benny Sudakov}
\thanks{Department of Mathematics, ETH Z\"urich, Switzerland. Email: {\tt \{zach.hunter, aleksa.milojevic, benjamin.sudakov\}@math.ethz.ch}. Research supported in part by SNSF grant 200021-228014.}

\begin{document}

\begin{abstract}
We give a simple proof of the recent remarkable exponential improvement for Ramsey lower bounds, obtained by Ma, Shen and Xie. Our key ingredient is an alternative construction based on Gaussian random graphs, which allows us to simplify their analysis significantly. As a consequence of this simpler analysis, we also obtain better quantitative bounds.
\end{abstract}

\maketitle

\section{Introduction}\label{sec:intro}

For positive integers $k, \ell\geq 1$, the \textit{Ramsey number} $R(\ell, k)$ is the smallest integer $n$ such that any edge-coloring of the complete graph $K_n$ using the colors red and blue either contains a copy of $K_\ell$ all of whose edges are red or a copy of $K_k$ all of whose edges are blue. These numbers play a central role in \textit{Ramsey theory}, which is itself an essential branch of combinatorics with a long and rich history, with hundreds of papers published each year. Arguably the most famous problem in Ramsey theory is to determine the growth of Ramsey numbers $R(\ell, k)$ when $k=\ell$, also known as the \textit{diagonal Ramsey numbers}, and more generally when $k=C\ell$. Without attempting to comprehensively survey all of the recent results, let us mention that the last couple of years have brought a series of exciting developments on many long-standing questions related to Ramsey numbers of cliques, see e.g. \cite{BBCGHMST, CGMS, CJMS25, CF21, HHKP25, MSX25, MV24, MV19} and the references therein.

The first quantitative upper bound on $R(\ell, k)$ comes from a classical paper of Erd\H{o}s and Szekeres \cite{ES35} from 1935, and states that $R(\ell, k)\leq \binom{\ell+k-2}{\ell-1}$, thus showing that $R(\ell, \ell)\leq 4^\ell$. After several important improvements by Thomason \cite{Tho88}, Conlon \cite{Con09} and Sah \cite{Sah23}, in a recent breakthrough paper Campos, Griffiths, Morris and Sahasrabudhe \cite{CGMS} showed that $R(\ell, \ell)\leq 4^{(1-\eps)\ell}$ for a positive constant $\eps>0$ (see also \cite{GNWW24} and \cite{BBCGHMST} for further developments).

The study of the lower bounds for the Ramsey numbers had at least equal, if not greater, influence on combinatorics. In the paper which introduced the probabilistic method to combinatorics, Erd\H{o}s \cite{Erd47} in 1947 gave an argument using random graphs which shows $R(\ell, \ell)\geq \frac{\ell}{e\sqrt{2}} 2^{\ell/2}$. Beyond a modest improvement \cite{Spe75} by a factor of $2$ using the Lov\'asz Local Lemma, the progress on this question was extremely slow and Erd\H{o}s' bound stays essentially unchanged to this day. Hence it was a great surprise when Ma, Shen and Xie \cite{MSX25} used a different random graph model to give an exponential improvement to Erd\H{o}s' lower bound for $R(\ell, k)$ when $k=C\ell$ for a constant $C>1$. To obtain this improvement, they used a geometrically defined random graph, whose vertices are random points on a high-dimensional sphere and the colors of the edges are determined by the distances between the points. Because of the complicated geometry of high-dimensional spheres, the analysis of this random graph model is long and technical, spanning many pages.

In this paper, we propose an alternative approach to this problem. Although we draw our inspiration from the work of Ma, Shen and Xie, we use random Gaussian vectors to define our graphs. Of course, in the regime of parameters considered here, the length of the Gaussian random vectors we choose is highly concentrated, so the resulting graph is very similar to the one obtained by placing random points on the sphere. However, the key advantage is that the Gaussian measure is a product measure, that is to say its coordinates are independent. This, along with a couple of further ideas, allows us to significantly simplify and shorten the analysis of such graphs, and to obtain tighter bounds.

In order to state our main theorem, let us briefly recall the classical lower bound on $R(\ell, k)$ coming from Erd\H{o}s' argument. If we color every edge of the graph red with probability $p$ or blue with probability $1-p$, independently of all other edges, then the probability that $\ell$ vertices form a red clique is exactly $p^{\binom{\ell}{2}}$, and the probability that $k$ vertices form a blue clique is $(1-p)^{\binom{k}{2}}$. Thus, the expected number of red $\ell$-cliques is $\binom{n}{\ell}p^{\binom{\ell}{2}}$, and the expected number of blue $k$-cliques is $\binom{n}{k}(1-p)^{\binom{k}{2}}$. If for some positive integer $n$ we have that $\binom{n}{\ell}p^{\binom{\ell}{2}}+\binom{n}{k}(1-p)^{\binom{k}{2}}<1$, then there exists a coloring in which both the number of red $\ell$-cliques and blue $k$-cliques is zero, and therefore we have $R(\ell, k)\geq n$. When $k=C\ell$, the probability $p$ which gives the best lower bound is denoted by $p_C$, and it can be found as the unique solution of the equation $C=\frac{\log p_C}{\log(1-p_C)}$. In this case, the corresponding lower bound is $R(\ell, k)\gtrsim p_C^{-\ell/2}$. In this paper we improve this bound as follows.

\begin{theorem}\label{thm:main}
For any $C>1$, there exists a constant $\eps=\eps(C)>0$ such that, for all sufficiently large $\ell$, we have
{\smaller[0.3]
\[R(\ell, C\ell)\geq (p_C^{-1/2}+\eps)^\ell,\]
}
where $p_C\in (0, 1/2)$ is the unique solution to the equation $C=\frac{\log p_C}{\log (1-p_C)}$.
\end{theorem}

Beyond its simplicity, another advantage of this analysis is that it allows us to obtain quantitative improvements over the work of Ma, Shen and Xie. Namely, as $C\to \infty$ and $p_C\to 0$, Ma, Shen and Xie obtain a gain of $\eps$ which goes to $0$ with $C\to \infty$. On the other hand, the argument we present here can be extended to show one can take $\eps=(e^{1/24}-1)p_C^{-1/2}$, thus showing $R(\ell, C\ell)\geq (e^{1/24}p_C^{-1/2})^\ell$ for all sufficiently large $C$. Let us also remark that when $C\to 1$, our proof gives the same quantitative behavior $\eps=\Omega((C-1)^2)$ as \cite{MSX25}. 
To focus on the simplest possible presentation, we present this proof in Appendix~\ref{sec:appendix_2} of this paper.

\subsection{Gaussian random geometric graph \texorpdfstring{$G(n, d, p)$}{G(n,d,p)}.}

Let us now formally define the model of the Gaussian random geometric graph, before saying a couple of words about how to analyze it. Let $n, d\geq 1$ be positive integers, and let $p\in (0, 1/2)$ be a probability. Also, let $c_p>0$ be the unique real number for which $\Pb[Z\leq -c_p]=p$, where $Z\sim \cN(0, 1)$ is the standard one-dimensional Gaussian.

\begin{definition}\label{def:RGG}
The vertices of the Gaussian random geometric graph $G(n, d, p)$ correspond to vectors $\bx_1, \dots, \bx_n$, independently sampled from the $d$-dimensional normal distribution $\cN(0, \frac{1}{d}I_d)$. Further, the vertices $i, j\in [n]$ are connected by an edge if $\langle \bx_i, \bx_j\rangle\geq -\frac{c_p}{\sqrt{d}}$.
\end{definition}

Note that given a graph $G\sim G(n, d, p)$, we can define a red/blue edge-coloring of the complete graph $K_n$, by declaring the edges of $G$ to be blue, and all other edges to be red.

Our notion of random geometric graphs differs slightly from the one commonly used in the literature. There are two main reasons. Firstly, it would be common to define the threshold $c_p$ such that $\langle \bx_i, \bx_j\rangle\geq -\frac{c_p}{\sqrt{d}}$ occurs exactly with probability $1-p$, when $\bx_i, \bx_j\sim \cN(0, \frac{1}{d}I_d)$, since the graph $G(n, d, p)$ would then have expected density \textit{exactly} $1-p$.

However, we opt for a simpler definition of $c_p$, and observe that the graph we define has expected density \textit{very close} to $1-p$ (in fact, the expected density converges to $1-p$ as $d$ increases). The reason is that the length of vectors $\bx_i$ is well concentrated around $1$, and so $\langle \bx_i, \bx_j\rangle$ is essentially the length of the projection of $\bx_j$ onto $\bx_i$, which is a Gaussian random variable of variance $\frac{1}{d}$, and as such exceeds the threshold $-c_p/\sqrt{d}$ with probability $1-p$.

While this first difference is mostly cosmetic, the second one is more substantial. It seems that random geometric graphs of density larger than $1/2$ (i.e. when the threshold is taken to be negative) have not been thoroughly studied in the literature, with most works focusing on the regime of sparse random geometric graphs. However, there is no obvious symmetry relating the two settings, such as taking the complements. Still, as \cite{MSX25} and this work demonstrate, it is worthwhile to study the dense random geometric graphs, especially since some of the techniques from the random geometric graph literature do not work in this setting anymore.

Let us also mention that the random geometric graphs have been thoroughly investigated in their own right. For the early study of random geometric graphs, mostly in low dimension, see Penrose's book \textit{Random Geometric Graphs} \cite{Pen03}. Some 15 years ago, there was a shift and the random geometric graphs in high dimension also started receiving considerable attention. In one of the first such works, Devroye, Gy\"orgy, Lugosi and Udina \cite{DGLU11} studied the clique numbers of random geometric graphs, showing that, for example, when $d\gg (\log n)^3$ then $\omega(G(n, d, p))\leq O(\log n)$ with high probability. They also studied the regime when $d\approx (\log n)^2$ (which is the regime relevant for us, since $\ell\approx \log n$ and $d\approx \ell^2$), although they only provided the upper bound $\omega(G(n, d, p))\leq O((\log n)^3)$ in this case. Thus, their results are not sufficient for the purposes of this paper, and we require a much more precise analysis of the clique numbers when $d$ is as low as $(\log n)^2$.

More generally, one can ask for which $d$ is the distribution of $G(n, d, p)$ close to that of $G(n, p)$ (in total variation distance, say), and this question received a lot of attention over the years (see e.g. \cite{BDER16}). However, this question is not the focus of our article and we will return to it in a future paper.

\subsection{Applying Gaussian random graphs to Ramsey number lower bounds}

Let us now discuss how the model $G(n, d, p)$ is used to prove Theorem~\ref{thm:main}. We define the edge-coloring of $K_n$ based on a sample $G\sim G(n, d, p)$ by coloring all edges present in $G$ blue and all missing edges red. We will choose $n$ to be the intended lower bound for the Ramsey number, i.e. $n=(p_C^{-1/2}+\eps)^\ell$ and $d=D^2\ell^2$, where $D$ is a large constant depending only on $C$ and $\ell \gg D, C$. Finally, the density $p$ of the red graph will be slightly larger than $p_C$, for the reasons we will explain below.

As usual, our goal will be to show that the expected number of red cliques of size $\ell$ and blue cliques of size $C\ell$ in this coloring is less than $1$. By linearity of expectation, this reduces to bounding the probability that a red clique of size $\ell$ or a blue clique of size $C\ell$ arises in the above coloring. Equivalently, the goal is to bound $\Pb[G(C\ell, d, p)\text{ is a clique}]$ and $\Pb[G(\ell, d, p)\text{ is an independent set}]$. Since these probabilities will play the central role in the proofs, let us fix the notation $P_{{\rm red}, r}=\Pb[G(r, d, p)\text{ is an independent set}]$ and $P_{{\rm blue}, r}=\Pb[G(r, d, p)\text{ is a clique}]$, where $r\geq 1$ is a positive integer. The bounds on $P_{\red, \ell}$ and $P_{\blue, r}$ which we obtain are summarized in the following proposition.

\begin{proposition}\label{prop:clique_probabilities}
Let $G\sim G(n, d, p)$ be a Gaussian random geometric graph, and let $a=\frac{e^{-c_p^2/2}}{\sqrt{2\pi}}$. If $d\geq D^2\ell^2$ where $D$ is sufficiently large compared to $C$, then
{\smaller[0.3]
\begin{align}
&P_{\red, \ell}\leq p^{\binom{\ell}{2}}\exp\bigg(-\frac{a^3}{p^3\sqrt{d}}\binom{\ell}{3}\Big(1+O_p\Big(\frac{1}{D}\Big)\Big)\bigg)\label{eq:clique_probabilities_red},\\
&P_{\blue, C\ell}\leq (1-p)^{\binom{C\ell}{2}}\exp\bigg(\frac{a^3}{(1-p)^3\sqrt{d}}\binom{C\ell}{3}\Big(1+O_{p, C}\Big(\frac{1}{D}\Big)\Big)\bigg)\label{eq:clique_probabilities_blue}.
\end{align}
}
\end{proposition}

In the above proposition, the notation $O_p(1/D)$ and $O_{p, C}(1/D)$ hides constants depending on $p$ and $C$, but not on $\ell$ or $d$. To avoid cluttering the formulas, we will often omit the subscripts $p, C$ and simply write $O(\cdot)$ to mean $O_{p, C}(\cdot)$. This is justified since we are assuming that $D$ is sufficiently large compared to $p, C$ (note that we will assume a different convention in the appendix, where the obtaining the precise dependence of $D$ on $p, C$ will be crucial). At this point, it is good to say a couple of words about the intuition behind the shape of these bounds. In the first order, the terms $p^{\binom{\ell}{2}}$ and $(1-p)^{\binom{C\ell}{2}}$ correspond to the behavior expected in an Erd\H{o}s-R\'enyi random graph which we already discussed above.

The corrections in formulas (\ref{eq:clique_probabilities_red}), (\ref{eq:clique_probabilities_blue}) account for the correlations between the edges of $G$. Intuitively, if the edges $ij$, $jk$ are colored blue, that means that the pairs of vectors $\bx_i, \bx_j$ and $\bx_j, \bx_k$ are close, and so $\bx_i$ and $\bx_k$ are more likely to be close as well, thus making the edge $ik$ slightly more likely. Hence, this positive correlation between neighboring blue edges is manifested in the fact that the correction in formula (\ref{eq:clique_probabilities_blue}) is larger than $1$, i.e. that the probability of $G(C\ell, d, p)$ being a clique is higher than in the Erd\H{o}s-R\'enyi random graph.

Let us make this vague intuition quantitative. If $\bx_i, \bx_j, \bx_k\sim \cN(0, \frac{1}{d}I_d)$ are independent, let
{\smaller[0.3]
\[\delta_b=\frac{\Pb[jk \text{ is blue}|ij, ik \text{ are blue}]}{\Pb[jk \text{ is blue}]}>1.\]
}
In our particular model, it is not very hard to calculate $\delta_b\approx 1+\frac{a^3}{(1-p)^3\sqrt{d}}$. With this definition, the probability of a triangle is in fact $(1-p)^3\delta_b$ instead of just $(1-p)^3$. Moreover, we should expect that
{\smaller[0.3]
\[P_{\blue, C\ell}\approx (1-p)^{\binom{C\ell}{2}}\delta_b^{\binom{C\ell}{3}},\]
}
where the factor $\delta_b$ appears once for every triangle $i j k$. Proving this formula precisely is the main content of Proposition~\ref{prop:clique_probabilities}. 

On the other hand, the real gain comes from the red cliques. Namely, if $ij, jk$ are red edges, then the vectors $\bx_i, \bx_k$ have very negative inner products with $\bx_j$. In other words, the vectors $\bx_i, \bx_k$ are very far from $\bx_j$, and as an exaggeration, one can imagine that they point in an almost opposite direction from $\bx_j$. Then, $\bx_i, \bx_k$ are most likely close together, making it unlikely that $ik$ is not an edge. Thus, having independent sets in Gaussian random geometric graphs is much less likely than in binomial random graph. This intuition can be made quantitative by defining $\delta_r={\Pb[jk \text{ is red}|ij, ik \text{ are red}]}/{\Pb[jk \text{ is red}]}$ and calculating $\delta_r=1-\frac{a^3}{p^3\sqrt{d}}$ and $P_{\red, \ell}\approx p^{\binom{\ell}{2}}\delta_r^{\binom{\ell}{3}}$ in a way paralleling the discussion above.

At this point, one might reasonably complain - the red cliques are less likely in our random coloring compared to the binomial random graph, but this comes at a price of blue cliques being more likely, and our goal is to avoid \textit{both} red and blue cliques. The final observation is that one might increase $p$ slightly (and consequently decrease the density blue edges), thus lowering the probability of blue cliques and raising the probability of red cliques. If one gets the balance just right, the gain coming from (\ref{eq:clique_probabilities_red}) outweighs the losses in (\ref{eq:clique_probabilities_blue}), and one obtains the graph in which both blue and red cliques are less likely compared to $G(n, 1-p_C)$. Let us formalize the last paragraph in the form of the following lemma.

\begin{lemma}\label{lemma:bookkeeping}
Let $C>1$ be a constant and let $p_C\in (0, 1/2)$ be the unique solution to $C=\frac{\log p_C}{\log (1-p_C)}$. If $D$ is sufficiently larger than $C$, there exist $\eps_1=\eps_1(C, D)>0$ and $p\in (0, 1/2)$ such that the clique probabilities in $G(n, d, p)$  with $d=D^2\ell^2$ are bounded by
{\smaller[0.3]
\[P_{{\rm red}, \ell}\leq \Big(p_C-\eps_1\Big)^{\binom{\ell}{2}} \text{ and } P_{{\rm blue}, C\ell}\leq \Big(1-p_C-\eps_1 \Big)^{\binom{C\ell}{2}}.\]
}
\end{lemma}
\begin{proof}[Proof of Lemma~\ref{lemma:bookkeeping} assuming Proposition~\ref{prop:clique_probabilities}.]
As we have indicated in the discussion preceding the lemma, in bounds (\ref{eq:clique_probabilities_red}), (\ref{eq:clique_probabilities_blue}) we have gained more in the probability $P_{{\rm red}, \ell}$ than we have lost in the bound on $P_{{\rm blue}, C\ell}$. This is justified in the following claim.
\begin{claim}\label{claim:inequality}
{\smaller[0.3]
\[\frac{a^3}{3p_C^2}>\frac{a^3C}{3(1-p_C)^2}.\]
}
\end{claim}
\begin{claimproof}
By canceling out $a^3/3$, we reduce the inequality to $1/p_C^{2}>C/(1-p_C)^2$, i.e. to showing that $(1-p_C)^2>Cp_C^2$. If we recall that $C=\frac{\log p_C}{\log (1-p_C)}=\frac{\log_2 1/p_C}{\log_2 1/(1-p_C)}$ and further rearrange the inequality, it reduces to
{\smaller[0.3]
\[(1-p_C)^2\log_2 \frac{1}{1-p_C}>p_C^2\log_2 \frac{1}{p_C}.\]
}
Thus, the goal is to show that the function $f(t)=(1-t)^2 \log_2 \frac{1}{1-t}-t^2 \log_2 \frac{1}{t}$ is positive for all $t\in (0, 1/2)$. It is not hard to see that $\lim_{t\to 0^+}f(t)=\lim_{t\to 1/2^-}f(t)=0$. Moreover, a simple calculation shows that the derivative of $f(t)$ is $f'(t)=\frac{1}{\ln 2}-2H(t)$, where $H(t)=-t\log_2 t-(1-t)\log_2(1-t)$ is the binary entropy function. Thus, $f''(t)=-2H'(t)$, and since $H$ is a strictly increasing function on $t\in (0, 1/2)$, we have $f''(t)<0$. Hence, $f$ is strictly concave and tends to $0$ at the endpoints of the interval $(0, 1/2)$, so it must be positive inside the interval $(0, 1/2)$.
\end{claimproof}

By continuity and since $p_C<1/2$, there exists an open neighborhood $(p_C-\eta, p_C+\eta)\subset (0, 1/2)$ of $p_C$ and a constant $\eps_0$ (where $\eta$ and $\eps_0$ depend only on $p_C$) such that
{\smaller[0.3]
\[\underbrace{\inf_{|p-p_C|< \eta}\frac{a^3}{3p^2}}_{\alpha}>\eps_0+\underbrace{\sup_{|p-p_C|< \eta}\frac{a^3C}{3(1-p)^2}}_{\beta}.\]
}

If we denote the above infimum and supremum by $\alpha$ and $\beta$, we note that both $\alpha$ and $\beta$ depend only on $p_C$ (and thus only on $C$). We also have that $\alpha>\eps_0+\beta$. We can use the Proposition~\ref{prop:clique_probabilities} to write
{\smaller[0.3]
\begin{align*}
    &P_{{\rm red}, \ell}^{1/\binom{\ell}{2}} \leq p \exp\Big(-\big(1+O(D^{-1})\big)\frac{\alpha \ell}{p\sqrt{d}}\Big)\leq p \exp\Big(-\big(1+O(D^{-1})\big)\frac{\alpha}{pD}\Big), \text{ and } \\
    &P_{{\rm blue}, C\ell}^{1/\binom{C\ell}{2}}\leq (1-p) \exp\Big(\big(1+O(D^{-1})\big)\frac{\beta \ell}{(1-p)\sqrt{d}}\Big)\leq (1-p) \exp\Big(\big(1+O(D^{-1})\big)\frac{\beta}{(1-p)D}\Big)
\end{align*}
}
Since $D$ is very large compared to $C$, we choose $|p-p_C|<\eta$, and so by using $e^t=1+t+O(t^2)$ for $|t|\leq 1$ we can write
{\smaller[0.3]
\begin{align*}
&P_{{\rm red}, \ell}^{1/\binom{\ell}{2}} \leq p\bigg(1-\frac{\alpha}{pD}+O(D^{-2})\bigg)=p-\frac{\alpha}{D}+O(D^{-2}), \text{ and }\\
&P_{{\rm blue}, C\ell}^{1/\binom{C\ell}{2}} \leq (1-p)\bigg(1+\frac{\beta}{(1-p)D}+O(D^{-2})\bigg)=1-p+\frac{\beta}{D}+O(D^{-2}).
\end{align*}
}
Hence, if we set $p=p_C+\frac{\alpha+\beta}{2D}$ (recall that $\alpha, \beta$ depend only on $p_C$) and let $D$ be sufficiently larger than $C$, we obtain
{\smaller[0.3]
\[P_{{\rm red}, \ell}\leq \Big(p_C-\frac{\alpha-\beta}{2D}+O(D^{-2})\Big)^{\binom{\ell}{2}} \text{ and } P_{{\rm blue}, C\ell}\leq \Big(1-p_C-\frac{\alpha-\beta}{2D}+O(D^{-2})\Big)^{\binom{C\ell}{2}}.\]
}
By setting $\eps_1<\eps_0/2D$ and letting $D$ be sufficiently large, we obtain precisely what we wanted.
\end{proof}

\begin{proof}[Proof of Theorem~\ref{thm:main} assuming Proposition~\ref{prop:clique_probabilities}.]
By Lemma~\ref{lemma:bookkeeping}, there exists some $\eps_1>0$ and $p\in (0, 1/2)$ for which the clique probabilities in $G(n, d, p)$ (where $d=D^2\ell^2$) satisfy
{\smaller[0.3]
\[P_{{\rm red}, \ell}\leq \Big(p_C-\eps_1\Big)^{\binom{\ell}{2}} \text{ and } P_{{\rm blue}, C\ell}\leq \Big(1-p_C-\eps_1 \Big)^{\binom{C\ell}{2}}.\]
}
Also, set $n=(p_C^{-1/2}+\eps)^\ell$ for some $\eps\ll \eps_1$, and let us show that the red/blue coloring of $K_n$ defined using the Gaussian random geometric graph has no red $\ell$-clique and no blue $C\ell$-clique with high probability.

By the union bound the probability that this coloring has a red $\ell$-clique is at most
{\smaller[0.3]
\[\binom{n}{\ell}\cdot P_{\red, \ell}\leq \frac{n^{\ell}}{\ell!}(p_C-\eps_1)^{\frac{\ell(\ell-1)}{2}}\leq \frac{1}{\ell!}\Big(n(p_C-\eps_1)^{\frac{\ell-1}{2}}\Big)^\ell=\frac{1}{\ell!}\Big((p_C^{-1/2}+\eps)^{\ell}(p_C-\eps_1)^{\frac{\ell-1}{2}}\Big)^\ell.\]
}
The crucial observation is that, inside the parentheses, we have an exponential term in $\ell$ with the base of the exponent being $(p_C^{-1/2}+\eps)(p_C-\eps_1)^{1/2}<1$. Also, note that the additive constant $-1/2$ in the exponent $\frac{\ell-1}{2}$ does not affect the calculation, since $\ell$ is very large compared to $p_C$. Hence, the probability that there exists a red $\ell$-clique is at most $\frac{1}{\ell!}$.

Similarly, the probability that the coloring contains a blue $C\ell$-clique is at most
{\smaller[0.3]
\[\binom{n}{C\ell}\cdot P_{\blue, C\ell}\leq \frac{n^{C\ell}}{(C\ell)!}(1-p_C-\eps_1)^{\frac{C\ell(C\ell-1)}{2}}\leq \frac{1}{(C\ell)!}\Big(n(1-p_C-\eps_1)^{\frac{C\ell-1}{2}}\Big)^{C\ell}.\]
}
Again, the base of the exponential term inside the parentheses is $(p_C^{-1/2}+\eps)(1-p_C-\eps_1)^{C/2}<1$ (which follows from $(1-p_C)^C=p_C$). Hence, the probability that there exists a blue $C\ell$-clique is at most $\frac{1}{(C\ell)!}$. This shows that there is an outcome without red $\ell$-cliques or blue $C\ell$-cliques, completing the proof.
\end{proof}

\paragraph{Organization.} In Section~\ref{sec:prelims}, we collect a number of well-known results about the Gaussian distribution, all of whose proofs are rather standard. Although we provide the complete proofs of these results in Appendix~\ref{sec:appendix_1}, we suggest the reader to skim this section by only reading the statements, since the heart of the paper lies in Sections~\ref{sec:perfect} and~\ref{sec:induction}.
In Section~\ref{sec:perfect}, we reduce bounding the probabilities $P_{\red, r}, P_{\blue, r}$ in Proposition~\ref{prop:clique_probabilities} to an easier problem of bounding the probability that random Gaussian vectors $\bx_1, \dots, \bx_r$ both form a monochromatic clique and have the inner products of expected size (i.e. that they form the so-called perfect sequence). Finally, in Section~\ref{sec:induction} we present the most important part of our argument, in which we give the remaining ingredients needed to prove Proposition~\ref{prop:clique_probabilities}.
We optimize our methods and give the proof of $R(\ell, C\ell)\geq ((1+\eps)p_C^{-1/2})^{\ell}$ in Appendix~\ref{sec:appendix_2}. 

\section{Preliminaries}\label{sec:prelims}

In this section, we will recall a collection of rather standard facts related to Gaussian random variables, which we mostly provide for reference.
We will begin by proving two concentration inequalities. Then, we will discuss truncated Gaussian random variables, and some of their properties. Finally, we conclude by bounding exponential moments of a certain quadratic function of truncated Gaussians, which will play an important role in our proofs later on. The first result we present is the concentration inequality for norms of high-dimensional Gaussian vectors.

\begin{lemma}\label{lemma:norm_concentration}
If $\bx\sim \cN(0, \frac{1}{d}I_d)$ and $\delta\in (0, 1)$, then $\Pb\big[\|\bx\|_2\in (1-\delta, 1+\delta)\big]\geq 1-2\exp(-\delta^2 d/10)$.
\end{lemma}

\noindent
Observe that $d\|\bx\|_2^2$ is a sum of squares of $d$ independent standard one-dimensional Gaussians, and therefore $d\|\bx\|_2^2$ follows the so-called $\chi_d^2$ distribution with $d$ degrees of freedom. Hence, Lemma~\ref{lemma:norm_concentration} follows as a simple application of the famous Laurent-Massart concentration inequality for $\chi^2_d$.

\begin{theorem}[Lemma 1 in \cite{LM}]\label{thm:laurent_massart}
Let $Y\sim\chi^2_d$ and let $t\ge0$. Then
{\smaller[0.3]
\[ \Pb\big[Y-d\ge 2\sqrt{dt}+2t\big]\le e^{-t} \quad\text{and}\quad \Pb\big[d-Y\ge 2\sqrt{dt}\big]\le e^{-t}. \]
}
\end{theorem}

Next, we give a one-sided tail bound for the length of a projection of a Gaussian vector to a lower-dimensional subspace. For a vector $\bx\in \bR^d$ and a subspace $W\subseteq \bR^d$, let $\pi_W(\bx)$ be the projection of $\bx$ to $W$.

\begin{lemma}\label{lemma:norm_projection}
Let $1\leq s\leq C\ell$ and let $W$ be any $s$-dimensional subspace of $\bR^d$. If $\bx\sim \cN(0, \frac{1}{d}I_d)$ is a random Gaussian vector and $\alpha\in (1, \infty)$ such that $\alpha^2\ell/8\geq s$, then
{\smaller[0.3]
\[\Pb\Big[\|\pi_W(\bx)\|_2\geq \frac{\alpha\sqrt{\ell}}{\sqrt{d}}\,\,\Big]\leq \exp(-\alpha^2\ell/8).\]
}
\end{lemma}


Since Gaussian random variables are the crucial players in this paper, let us denote by $\phi, \Phi$ the probability density function and the cumulative distribution function of the standard normal $\cN(0, 1)$. Among their many useful properties, we will often use the following bound on the Mills ratio: for every negative $t<0$ we have
{\smaller[0.3]
\begin{equation}\label{eqn:mills_ratio}
|t|\leq \frac{\phi(t)}{\Phi(t)}\leq |t|+\frac{1}{|t|}.
\end{equation}
}
This follows from a simple integration by parts, see e.g. \cite{Gor41}. Another standard fact related to the Gaussian distribution is that the function $\Phi(t)$ is log-concave (i.e. $\log \Phi(t)$ is a concave function). This will be quite useful, since it allows us to derive that $\log \Phi(t+\eps)\leq \log \Phi(t)+\eps (\log \Phi(t))'$, which can also be read as
{\smaller[0.3]
\begin{equation}\label{eqn:log_concavity}
    \Phi(t+\eps)\leq \Phi(t)e^{\eps\phi(t)/\Phi(t)}.
\end{equation}
}

Truncated Gaussian random variables will also play an important role in the paper. If $Z\sim \cN(0, \sigma^2)$ is a one-dimensional Gaussian and $(a, b)\subseteq \bR$ an interval, the \textit{truncated Gaussian} is the distribution of $Z$ conditioned on $Z\in (a, b)$. Our distribution will be truncated only on one side, i.e. we will only consider cases where $a=-\infty$ or $b=+\infty$. If $a\neq -\infty$, we have a lower truncated variable, and if $b\neq+\infty$ we have an upper truncated variable. We begin our discussion of truncated Gaussians by computing their expectation.

\begin{lemma}\label{lemma:truncated_expectation}
Let $b, \eps$ be real numbers and let $X\sim \cN(0, \frac{1}{d})$. Then, we have the following four formulas
{\smaller[0.3]
\begin{align*}
    &\bE\bigg[X\Big| X\leq \frac{b}{\sqrt{d}}\bigg]=-\frac{\phi(b)}{\Phi(b)\sqrt{d}} \quad \,\,\,\,\text{ and }\bE\bigg[X\Big| X\leq \frac{b+\eps}{\sqrt{d}}\bigg]=-\frac{\phi(b)}{\Phi(b)\sqrt{d}}+O\bigg(\frac{\eps}{\sqrt{d}}\bigg),\\
    &\bE\bigg[X\Big| X\geq \frac{b}{\sqrt{d}}\bigg]=\frac{\phi(b)}{(1-\Phi(b))\sqrt{d}}\text{ and } \bE\bigg[X\Big| X\geq \frac{b+\eps}{\sqrt{d}}\bigg]=\frac{\phi(b)}{(1-\Phi(b))\sqrt{d}}+O\bigg(\frac{\eps}{\sqrt{d}}\bigg).
\end{align*}
}
\end{lemma}

Let us close this section by showing how to bound the exponential moments of a certain quadratic sum of truncated Gaussian random variables. Let us begin by recalling the notion of subgaussian random variables. They will be important to us since they provide a good framework for bounding exponential moments, and since truncated Gaussians are subgaussian.

A random variable $X$ is called \textit{subgaussian} if there exists a positive number $\sigma^2$ such that
{\smaller[0.3]
\[\bE\big[\exp\big(\lambda (X-\bE X)\big)\big]\leq \exp\big(\sigma^2\lambda^2/2\big) \text{ for all }\lambda\in \bR.\]
}
The smallest such constant $\sigma^2$ is called the \textit{variance proxy}. The paper \cite{BMA24} shows that the optimal variance proxy of a truncated Gaussian is at most the variance of the corresponding Gaussian.

Subgaussian random variables possess many useful properties, and here we will survey only those which we will need in the paper. For a more thorough treatment, we suggest to consult Vershynin's book \cite{Ver18}. A basic observation, which follows directly from the definition, is that if $X_1, \dots, X_k$ are independent subgaussians with variance proxies $\sigma_1^2, \dots, \sigma_k^2$, then $X_1+\dots+X_k$ is also subgaussian with variance proxy $\sigma_1^2+\dots+\sigma_k^2$. The most important property of subgaussians will be the following bound on the exponential moments of their squares.

\begin{lemma}\label{lemma:exponential_square_moments}
Let $X$ be a subgaussian of variance proxy $\sigma^2$ and with $\bE[X]=0$, and let $0\leq \lambda< \frac{1}{2\sigma^2}$. Then
{\smaller[0.3]
\[\bE[e^{\lambda X^2}]\leq 1+\frac{4\lambda \sigma^2}{1-2\lambda\sigma^2}.\]
}
\end{lemma}

Finally, the reason we introduced subgaussian random variables is to prove the following lemma, which says that the exponential moments of a certain quadratic sum of subgaussian random variables are essentially controlled by the expectation of this sum.

\begin{lemma}\label{lemma:exponential_moment_special_case}
Let $X_1, \dots, X_k$ be independent subgaussian random variables with variance proxy $1/d$, and let $\lambda\in (-\infty, \infty)$ such that $d\geq 4|\lambda| k$. If $S=\sum_{1\leq i<j\leq k}X_iX_j$, then
{\smaller[0.3]
\[\bE[e^{\lambda S}]\leq \exp\bigg( \lambda \bE[S] + \frac{\lambda^2k^2}{d}\sum_{j=1}^k (\bE X_j)^2 +\frac{4|\lambda| k}{d}\bigg).\]
}
\end{lemma}

\section{Perfect sequences}\label{sec:perfect}

In this section, we will reduce the problem of bounding $P_{\red, \ell}, P_{\blue, C\ell}$ to a simpler problem. Recall that in general, for $1\leq r\leq C\ell$, $P_{\red, r}$ denotes the probability that $r$ independent Gaussian vectors $\bx_1, \dots, \bx_r$ form a red clique. Observe that such random vectors are typically almost orthogonal to each other, because of Lemma~\ref{lemma:norm_projection}.

It turns out that it is much simpler to analyze the behavior of the sequence $\bx_1, \dots, \bx_r$ if indeed \textit{all} inner products $\langle \bx_i, \bx_j\rangle$ are small in absolute value. In other words, it is much easier to analyze the probability that a \textit{typical} sequence $\bx_1, \dots, \bx_r$ forms a red clique than a general one, and in both cases we essentially get the same answer.

Following Ma, Shen and Xie \cite{MSX25}, who used a very similar notion, we say that $\bx_1, \dots, \bx_r$ is a \textit{perfect sequence} if for all $1\leq i\leq r$ we have:
{\smaller[0.3]
\[\|\bx_i\|_2\in (1-\delta, 1+\delta) \quad\text{ and }\quad \|\pi_{{\rm span} \{\bx_1, \dots, \bx_{i-1}\}} (\bx_i)\|_2\leq \frac{\alpha\sqrt{\ell}}{\sqrt{d}},\]
}
where $\alpha=10\sqrt{C\log(10/p)}$ and $\delta=\alpha d^{-1/4}$.

Let $P_{\red, r}^*$ be the probability that $\bx_1, \dots, \bx_r$ form a red clique and a perfect sequence, and let $P_{\blue, r}^*$ be defined analogously.
In light of this definition, we now have two tasks - we need to bound $P_{\red, r}^*$ and we need to show that $P_{\red, r}$ and $P_{\red, r}^*$ are essentially the same. The first task will be performed via the following proposition.

\begin{proposition}\label{prop:perfect_probabilities}
Let $C>1, p\in (0, 1/2)$ be real numbers and $\ell, d=D^2\ell^2$ be large positive integers, where $D$ is sufficiently larger than $C$. For every integer $3\leq r\leq C\ell$, we have
{\smaller[0.3]
\begin{align}
    &P_{\red, r}^*\leq p^{\binom{r}{2}}\exp\bigg(-\big(1+O_p(D^{-1})\big)\frac{a^3}{p^3\sqrt{d}}\binom{r}{3}\bigg)\label{eq:perfect_probabilities_red}\\
    &P_{\blue, r}^*\leq (1-p)^{\binom{r}{2}}\exp\bigg(\big(1+O_p(D^{-1})\big)\frac{a^3}{(1-p)^3\sqrt{d}}\binom{r}{3}\bigg),\label{eq:perfect_probabilities_blue}
\end{align}
}
where $a=\phi(c_p)=\frac{e^{-c_p^2/2}}{\sqrt{2\pi}}$.
\end{proposition}

The proof of Proposition~\ref{prop:perfect_probabilities} is the heart of our argument, and it will be presented in the next section. In this section, however, we will show how to get the bounds on $P_{\red, r}, P_{\blue, r}$ from (\ref{eq:perfect_probabilities_red}), (\ref{eq:perfect_probabilities_blue}).

\begin{proof}[Proof of Proposition~\ref{prop:clique_probabilities} based on Proposition~\ref{prop:perfect_probabilities}.]
We will present a bound only for $P_{\red, r}$, since the bound on $P_{\blue, r}$ can be derived in a completely analogous fashion. Recall, we aim to bound $P_{\red, r}$, the probability that $\bx_1, \dots, \bx_r$ form a red clique.

Let us begin the proof by describing a procedure which extracts a perfect sequence $(\by_1, \dots, \by_t)$ from $(\bx_1, \dots, \bx_r)$. To do this, we go through the sequence $(\bx_1, \dots, \bx_r)$ in order. In the $i$-th step, we consider the element $\bx_i$, and assuming that we have defined $\by_1,\dots, \by_j$ so far, we check whether $\|\bx_i\|_2\in (1-\delta, 1+\delta)$ and $\|\pi_{{\rm span}\{\by_1, \dots, \by_j\}} (\bx_i)\|_2\leq \frac{\alpha \sqrt{\ell}}{\sqrt{d}}$. If both inequalities hold, we set $\by_{j+1}=\bx_i$, and otherwise we simply throw $\bx_i$ away.

Let $I$ be the set of indices $i$ for which $\bx_i$ was included in the sequence $(\by_1, \dots, \by_t)$. By definition, $(\by_1, \dots, \by_t)$ is a perfect sequence. Also, observe that $(\bx_1, \dots, \bx_r)$ is perfect if and only if $I=[r]$.

To bound the probability that $(\bx_1, \dots, \bx_r)$ forms a red $\ell$-clique, we use a union bound over all possible sets $I$ produced by the above procedure. If we let $p_I=\Pb\big[\bx_1, \dots, \bx_r$ form a red clique and the above procedure gives the set $I\big]$, then
{\smaller[0.3]
\begin{align}\label{eqn:union_bound_perfect}
P_{\red, r}&=\Pb[\bx_1, \dots, \bx_r \text{ form a red clique}]= \sum_{I\subseteq [r]} p_I.
\end{align}
}
As suggested above, the main term of the above sum comes from $I=[r]$. In this case, we are evaluating the probability that the sequence $\bx_1, \dots, \bx_r$ forms a red clique and is perfect, which equals $P_{\red, r}^*$ by definition. Let us now bound the contribution of terms coming from other sets $I$.

If $i\notin I$, then either $\|\bx_i\|_2\notin (1-\delta, 1+\delta)$ or $\big\|\pi_{{\rm span}\{\bx_1, \dots, \bx_{i-1}\}}\bx_i\big\|_2\geq \|\pi_{{\rm span}\{\by_1, \dots, \by_j\}} (\bx_i)\|_2
\geq \frac{\alpha\sqrt{\ell}}{\sqrt{d}}$, where we used that enlarging the subspace only increases the length of the projection. The first event happens with probability at most $2e^{-\delta^2d/10}\leq e^{-\Omega(\sqrt{d})}\ll (p/10)^{10C\ell}$ (due to Lemma~\ref{lemma:norm_concentration}). The second of these events also happens with probability at most $\exp(-\alpha^2\ell/8)$, due to Lemma~\ref{lemma:norm_projection} and the fact that $\alpha^2\ell/8\geq 100C\log (10/p)\ell/8>C\ell\geq s$. Since $\alpha=10\sqrt{C\log(10/p)}$, we have $\exp(-\alpha^2\ell/8)\leq (p/10)^{10C\ell}$. Finally, note that both of these statements are true even conditionally on $\bx_1, \dots, \bx_{i-1}$. Thus, we can write the following
{\smaller[0.3]
\begin{align*}
    p_I\leq \Pb\big[(\bx_i)_{i\in I}&\text{ forms a red clique and is perfect}\big]\\
    &\cdot \Pb\big[\|\bx_i\|_2\notin(1-\delta, 1+\delta)\text{ or }\big\|\pi_{{\rm span}\{\bx_1, \dots, \bx_{i-1}\}}(\bx_i)\big\|_2\geq \frac{\alpha\sqrt{\ell}}{\sqrt{d}}\text{ for all }i\notin I\big].
\end{align*}
}

The first probability is $P_{\red, |I|}^*$, while the second one is at most $\big(p/10\big)^{10C\ell\cdot (r-|I|)}$, due to the independence of $\bx_1, \dots, \bx_r$. Hence, we have
$p_I\leq P_{\red, |I|}^*\big(p/10\big)^{10 C\ell\cdot (r-|I|)}$. Let $r-|I|=u$.

We now use the bounds on $P_{\red, r-u}^*$ coming from Proposition~\ref{prop:perfect_probabilities}
{\smaller[0.3]
\begin{align*}
  P_{\red, r-u}^*&\leq p^{\binom{r-u}{2}}\exp\bigg(-\big(1+O(D^{-1})\big)\frac{a^3}{p^3\sqrt{d}}\binom{r-u}{3}\bigg)
\end{align*}
}
Observe now that
{\smaller[0.3]
\[\Big(\frac{p}{10}\Big)^{10 C\ell\cdot u}\leq p^{\binom{r}{2}-\binom{r-u}{2}}\cdot 2^{-C\ell u}\exp\bigg(-\big(1+O(D^{-1})\big) \frac{a^3}{p^3 \sqrt{d}}\Big(\binom{r}{3}-\binom{r-u}{3}\Big)\bigg),\]
}
since $\binom{r}{2}-\binom{r-u}{2}\leq ru\leq C\ell u$ and $\frac{a^3}{p^3\sqrt{d}}\Big(\binom{r}{3}-\binom{r-u}{3}\Big)\leq \frac{a^3}{p^3}\frac{ur^2}{\sqrt{d}}\leq \frac{a^3}{p^3}\frac{u(C\ell)^2}{D\ell}\leq C\ell u$ (recall that $D$ is sufficiently large as a function of $p, C$).
So,
{\smaller[0.3]
\[p_I\leq P_{\red, r-u}^*\cdot \Big(\frac{p}{10}\Big)^{10 C\ell\cdot u}\leq 2^{-2C\ell u} p^{\binom{r}{2}}\exp\bigg(-\big(1+O(D^{-1})\big)\frac{a^3}{p^3\sqrt{d}}\binom{r}{3}\bigg).\]
}

Let us now split the sum in (\ref{eqn:union_bound_perfect}) depending on the size of $I$. Recalling that $u=r-|I|$, we have
{\smaller[0.3]
\begin{align*}
    P_{\red, r}\leq p^{\binom{r}{2}}\exp\bigg(-\big(1+O(D^{-1})\big)\frac{a^3}{p^3\sqrt{d}}\binom{r}{3}\bigg)\sum_{u=0}^{r} \binom{r}{u} 2^{-2C\ell u}.
\end{align*}
}
Since $r\leq C\ell$, we easily see that $\sum_{u=0}^{r} \binom{r}{u} 2^{-2C\ell u}= (1+2^{-C\ell})^{r}\leq \exp\big(r/2^{C\ell}\big)\le 2$, and so
{\smaller[0.3]
\begin{align*}
    P_{\red, r}\leq  2p^{\binom{r}{2}}\exp\Big(-\big(1+O(D^{-1})\big)\frac{a^3}{p^3\sqrt{d}}\binom{r}{3}\Big).
\end{align*}
}
This suffices to complete the proof, by absorbing the constant factor $2$ into the $O_p(\frac{r^3}{D\sqrt{d}})$ error term in the exponent.
\end{proof}

\section{Changing the perspective}\label{sec:induction}

Our goal in this section is to bound the probability that $G\sim G(r, d, p)$ is a clique or an independent set and that the Gaussian vectors $\bx_1, \dots, \bx_r$ corresponding to its vertices form a perfect sequence, i.e. to bound $P_{\red, r}^*, P_{\blue, r}^*$. In other words, we will prove Proposition~\ref{prop:perfect_probabilities}. The first step in doing this will be to change coordinates to a more convenient basis, as follows.

For a sequence of $r$ points $\bx_1, \dots, \bx_r\sim \cN(0, \frac{1}{d}I_d)$, let us reveal the subspaces ${\rm span}\{\bx_1\}, {\rm span}\{\bx_1, \bx_2\}$, $\dots, {\rm span}\{\bx_1, \dots, \bx_r\}$. We may rotate the setup so that ${\rm span}\{\bx_1\}$ is aligned with the first basis vector $\be_1$, ${\rm span}\{\bx_1, \bx_2\}$ is the span of first two basis vectors $\be_1, \be_2$, and so on, until ${\rm span}\{\bx_1, \dots, \bx_r\}={\rm span}\{\be_1, \dots, \be_r\}$. After the rotation, the vector $\bx_i$ will have the first $i-1$ coordinates distributed as independent $\cN(0, \frac{1}{d})$ Gaussians, while the last coordinate will follow the distribution $\sqrt{\frac{1}{d}\chi_{d-i+1}^2}$. Recall that the $\chi_k^2$ distribution is defined as the sum of squares of $k$ one-dimensional standard normal random variables. The mentioned rotation can be obtained by applying the Gram-Schmidt process to the sequence $\bx_1, \dots, \bx_r$, and the set of vectors obtained after the rotation is sometimes known as the Bartlett decomposition.

\begin{definition}
Let $r$ be a positive integer with $2\leq r\leq d$. For $1\leq i\leq r$, let us define a random vector $\by_i\in \bR^d$ as follows: the first $i-1$ coordinates of $\by_i$ are Gaussian $\cN(0, \frac{1}{d})$, the $i$-th coordinate follows the distribution $\sqrt{\frac{1}{d}\chi_{d-i+1}^2}$, and all further coordinates are zero. Moreover, all of the above coordinates are sampled independently.
\end{definition}

\begin{lemma}
Let $d\geq r\geq 1$ be positive integers, let $\by_1, \dots, \by_r$ be sampled as above, and let $\bx_1, \dots, \bx_r\sim \cN(0, \frac{1}{d}I_d)$ be independent. Then, the collection of inner products $\big\{\langle \bx_i, \bx_j\rangle\big\}_{1\leq i, j\leq r}$ has the same joint distribution as $\big\{\langle \by_i, \by_j\rangle\big\}_{1\leq i, j\leq r}$.

In particular, if $G'$ is a random graph on the vertex set $[r]$ where the vertices $i,j$ are adjacent if $\langle \by_i, \by_j\rangle\geq -\frac{c_p}{\sqrt{d}}$, then $G'$ follows the same distribution as the random graph $G(r, d, p)$.
\end{lemma}
\begin{proof}
Let us denote the standard basis of $\bR^d$ by $\be_1, \dots, \be_d$. Let us sample $\bx_1, \dots, \bx_r\sim \cN(0, \frac{1}{d}I_d)$ independently and consider the following rotation. Let $O$ be a $d\times d$ orthogonal matrix such that ${\rm span}\{O\bx_1\}={\rm span}\{\be_1\}$, ${\rm span}\{O\bx_1, O\bx_2\}={\rm span}\{\be_1, \be_2\}$, $\cdots$, ${\rm span}\{O\bx_1, \dots, O\bx_r\}={\rm span}\{\be_1, \dots, \be_r\}$. Note that such a matrix exists whenever $\bx_1, \dots, \bx_r$ are linearly independent (which happens with probability $1$), and can be explicitly constructed using the Gram-Schmidt algorithm.

Let $\bz_i=O\bx_i$, and we claim that $(\bz_1, \dots, \bz_r)$ follows the same distribution as $(\by_1, \dots, \by_r)$, which is sufficient to conclude the proof (since $O$ does not change the inner products).

Observe first that $\bz_i(j)=0$ when $j>i$, since $\bz_i$ lies in the span of $\be_1, \dots, \be_i$. Further, note that $\pi_{{\rm span}\{\be_1, \dots, \be_{i-1}\}}(\bz_i)$ is the same as the projection of $\bx_i$ to ${\rm span}\{\bx_1, \dots, \bx_{i-1}\}$, which is simply a $(i-1)$-dimensional Gaussian distribution with covariance matrix $\frac{1}{d}I_{i-1}$. Thus, every coordinate of $\pi_{{\rm span}\{\be_1, \dots, \be_{i-1}\}}(\bz_i)$ has a Gaussian distribution with variance $\frac{1}{d}$, as needed. Finally, $\bz_i(i)$ is the residual length of $\bx_i$ which does not lie in the subspace ${\rm span} \{\bx_1, \dots, \bx_{i-1}\}$. However, since the Gaussian distribution is rotationally invariant, this is simply the length of a $(d-i+1)$-dimensional Gaussian, which is distributed as $\sqrt{\frac{1}{d}\chi_{d-i+1}^2}$. This completes the proof.
\end{proof}

Although we have initially defined vectors $\by_1, \dots, \by_r$ as elements of $\bR^d$, they are always zero on coordinates $r+1, \dots, d$, and so we might as well write $\by_1, \dots, \by_r\in \bR^r$. The main advantage of working in this new basis is that it simplifies the calculation of the inner products. More precisely, for $i<j$, we have
{\smaller[0.3]
\[\langle \by_i, \by_j\rangle=\by_j(i)\by_i(i) + \sum_{k=1}^{i-1}\by_i(k)\by_j(k).\]
}
Recall that $\by_i(k)$ denotes the $k$-th coordinate of $\by_i$. Since each coordinate $\by_i(k), \by_j(k)$ is expected to have size roughly $1/\sqrt{d}$ and $\by_i(i)\approx 1$, we see that the leading term of the above expression is $\by_j(i)$. In other words, whether the edge $ij$ is present or not is essentially determined by $\by_j(i)$, with small corrections coming from the other coordinates $\by_i(k), \by_j(k)$. Throughout the proof, we find it useful to arrange the vectors $\by_1, \dots, \by_r$ into a lower-triangular matrix $M$ whose rows are $\by_1, \dots, \by_r$.
{\smaller[0.3]
\[
\begin{array}{|c|c|c|c|}
\cline{1-1}
\by_1(1)   \\ \cline{1-2}
\by_2(1) & \by_2(2)   \\ \cline{1-3}
\by_3(1) & \by_3(2) & \by_3(3) \\ \cline{1-4}
\by_4(1) & \by_4(2) & \by_4(3) & \by_4(4) \\ \cline{1-4}
\end{array}
\]
}

Since $\by_1, \dots, \by_r$ is obtained from $\bx_1, \dots, \bx_r$ by a rotation, the notion of a perfect sequence remains unchanged. In particular, this means that each vector $\by_i$ has length approximately $1$, and that the most of this length comes from the diagonal entry $\by_i(i)$. More precisely, we have the following deterministic statement.

\begin{claim}\label{claim:diagonal_bound}
If $\by_1, \dots, \by_r\in \bR^r$ is a perfect sequence, then $M_{i, i}=\by_i(i)\in (1-2\delta, 1+\delta)$ for all $1\leq i\leq r$.
\end{claim}
\begin{proof}
We have $M_{i, i}^2=\by_i(i)^2=\|\by_i\|_2^2-\|\pi_{i-1}(\by_i)\|_2^2\geq (1-\delta)^2-\frac{\alpha^2\ell}{d}$, where we have used that $\by_1, \dots, \by_r$ is perfect in the last inequality. Since $\delta=\alpha d^{-1/4}$, we have $\alpha^2\ell/d\leq \alpha^2/\sqrt{d}=\delta^2$, and so $M_{i, i}^2\geq (1-\delta)^2-\delta^2=1-2\delta$, implying that $M_{i, i}\geq 1-2\delta$. On the other hand, we have $M_{i, i}\leq \|\by_i\|\leq 1+\delta$.
\end{proof}

The entries on the diagonal will not play an important role in our proof of Proposition~\ref{prop:perfect_probabilities}. Thus, we will condition on the specific outcome of $M_{1, 1}, \dots, M_{r, r}$ and assume each diagonal entry is of size $1\pm 2\delta$. Furthermore, for the sake of shortening the formulas, we will not specify the conditioning on the diagonal entries explicitly.

The main idea in the proof of Proposition~\ref{prop:perfect_probabilities} is to bound the probability that $\by_{s+1}, \dots, \by_r$ is a clique or an independent set, for every $0\leq s\leq r-1$, conditionally on the first $s$ columns of $M$ (note that if we condition on these columns, the connections between vectors $\by_1, \dots, \by_s$ and all other vectors are already determined). The bound will then be proved by induction on $s$, starting from $s=r-1$ and going all the way down to $s=0$. However, before we do that, let us introduce some terminology which will be used in the proof.

For $1\leq i<j\leq r$, let $E_{ij}$ be the event that $\langle \by_i, \by_j\rangle\geq -c_p/\sqrt{d}$, and let $\overline{E}_{ij}$ be its complement, i.e. the event that $\langle \by_i, \by_j\rangle< -c_p/\sqrt{d}$. Let $C_s=\bigwedge_{s< i<j\leq r} E_{ij}$ be the event that the vertices $s+1, \dots, r$ induce a clique in $G(r, d, p)$ and let $I_s=\bigwedge_{s< i<j\leq r} \overline{E}_{ij}$ be the event they induce an independent set. Finally, let $B_r$ denote the event that $\by_1, \dots, \by_r$ is a perfect sequence. This event implies that $\|\pi_i(\by_{i+1})\|_2\leq \frac{\alpha\sqrt{\ell}}{\sqrt{d}}$ for all $1\leq i\leq r-1$.

Also, for $1\leq s\leq i\leq r$, we write $\pi_s(\by_i)$ for the projection of $\by_i$ to the first $s$ coordinates (which is a shorthand for $\pi_{{\rm span}\{\by_1, \dots, \by_s\}}(\by_i)$). Further, let $M_i$ denote the set of entries in the $i$-th column of the matrix $M$ strictly below the diagonal.

\begin{proposition}\label{prop:induction_1}
Let $C>1, p\in (0, 1)$ be real numbers and $\ell, d=D^2\ell^2$ be large positive integers, where $D$ is sufficiently larger than $C$. Also, let $s, r$ be nonnegative integers such that $1\leq s+1\leq r\leq C\ell$. Suppose that the first $s$ columns of $M$ are fixed, and denote them by $M[s]=(M_1, \dots, M_{s})$. Then, the probabilities over the random choice of the remaining columns $M_{s+1}, \dots, M_r$ of the events $C_{s}\wedge B_r, I_{s}\wedge B_r$ are at most
{\smaller[0.3]
\begin{align}
    &\Pb\Big[I_s\wedge B_r\Big| M[s]\Big]\leq p^{\binom{r-s}{2}}\exp\bigg(\!-\!\frac{a\sqrt{d}}{p}\!\!\!\!\sum_{s<i< j\leq r}\!\!\!\!\langle \pi_s(\by_i), \pi_s(\by_j)\rangle \!-\!\frac{a^3}{p^3\sqrt{d}}\binom{r\!-\!s}{3}\!+\!O\Big(\frac{(r\!-\!s)^4}{d}\!+\!(r\!-\!s)\Big)\bigg)\label{eq:induction_bound_red}\\
    &\Pb\Big[C_s\wedge B_r\Big| M[s]\Big]\leq (1\!-\!p)^{\binom{r-s}{2}}\exp\bigg(\frac{a\sqrt{d}}{1\!-\!p}\!\!\sum_{s<i< j\leq r}\!\!\langle \pi_s(\by_i), \pi_s(\by_j)\rangle\!+\!\frac{a^3}{(1\!-\!p)^3\sqrt{d}}\binom{r\!-\!s}{3}  \label{eq:induction_bound_blue} \\
    &\hspace{11.0cm}\!+\!O\Big(\frac{(r\!-\!s)^4}{d}\!+\!(r\!-\!s)\Big)\bigg) \notag.
\end{align}
}
\end{proposition}
\begin{proof}
We argue by reverse induction on $s$. Note that for $s=r-1$, the statement is vacuous, since there the whole matrix is revealed and the product on the right hand side of Equation (\ref{eq:induction_bound_red},\ref{eq:induction_bound_blue}) is empty. Thus, we may use this as a basis of the induction. Let us first focus on bounding the probability of $C_s\wedge B_r$.

Suppose now that the statement holds for $s$ and let us prove it for $s-1$, where only the columns $M[s-1]=(M_1, \dots, M_{s-1})$ are fixed. We will reduce the inductive step to two auxiliary claims, which we will then prove separately. To do this, we sample the column $M_s$ by setting all of its entries below the diagonal to be an independent Gaussian $\cN(0, \frac{1}{d})$. This is sufficient to determine whether or not the vertex $s$ is connected to the vertices $s+1, \dots, r$. Let us begin by estimating this probability in our first auxiliary claim.

\begin{claim}\label{claim:auxiliary_1}
Given $M_1, \dots, M_{s-1}$, the probability (over the random choice of $M_s$) that $s$ is connected to all of $s+1, \dots, r$ is at most
{\smaller[0.3]
\begin{equation}\label{eq:auxiliary_1_blue}
    \Pb\bigg[\bigwedge_{s<i\leq r}E_{s, i}\bigg|M[s\!-\!1]\bigg]\leq (1\!-\!p)^{r-s}\exp\bigg(\frac{a\sqrt{d}}{1\!-\!p}\sum_{i=s+1}^r \langle \pi_{s-1}(\by_i), \pi_{s-1}(\by_s)\rangle\!+\!O\big((r\!-\!s)\delta\big)\bigg).
\end{equation}
}
Also, the probability that $s$ is not connected to any of $s+1, \dots, r$ is at most
{\smaller[0.3]
\begin{equation}\label{eq:auxiliary_1_red}
    \Pb\bigg[\bigwedge_{s<i\leq r}\overline{E}_{s, i}\bigg|M[s\!-\!1]\bigg]\leq
    p^{r-s}\exp\bigg(\!-\!\frac{a\sqrt{d}}{p}\sum_{i=s+1}^r \langle \pi_{s-1}(\by_i), \pi_{s-1}(\by_s)\rangle\!+\!O\big((r\!-\!s)\delta\big)\bigg).
\end{equation}
}
\end{claim}

Although the proof of Claim~\ref{claim:auxiliary_1} is quite simple, we postpone it for later. For each outcome of $M_s$, the induction hypothesis provides a bound on the probability of $C_s\wedge B_r$ over the random choice of $M_{s+1}, \dots, M_r$. Thus, in order to perform an induction step, we will integrate this bound over all outcomes of $M_s$ satisfying the event $C_s\wedge B_r$. At this point, it is important to observe that $C_{s-1}=C_{s}\wedge \bigwedge_{s<i\leq r}E_{si}$, i.e. that vertices $s, \dots, r$ form a clique precisely when vertices $s+1, \dots, r$ form a clique and $s$ is connected to all vertices among $s+1, \dots, r$. Hence, by the law of total probability, we can write
{\smaller[0.3]
\begin{align}\label{eq:total_probability_blue}
\Pb\Big[C_{s\!-\!1}\wedge B_r&\Big| M[s\!-\!1]\Big] =\bE_{M_s}\bigg[\Pb\Big[C_{s}\wedge B_r\Big| M_{s}\Big]\bigg| \bigwedge_{s<i\leq r}E_{s, i}, M[s\!-\!1]\bigg]\Pb\bigg[\bigwedge_{s<i\leq r}E_{s, i}\bigg|M[s\!-\!1]\bigg].
\end{align}
}

Since we have already estimated the latter term, we focus on the first one. We use the induction hypothesis to upper bound $\Pb\Big[C_{s}\wedge B_r\Big| M_{s}\Big]$ as follows
{\smaller[0.3]
\begin{equation}\label{eq:induction_hypothesis_blue}
    \Pb\Big[C_{s}\wedge B_r\Big| M_{s}\Big]\leq (1\!-\!p)^{\binom{r-s}{2}}\exp\bigg(\frac{a\sqrt{d}}{1\!-\!p}\!\!\sum_{s<i< j\leq r}\!\!\!\!\langle \pi_s(\by_i), \pi_s(\by_j)\rangle + \frac{a^3}{(1\!-\!p)^3\sqrt{d}}\binom{r\!-\!s}{3}\!+\!O\Big(\frac{(r\!-\!s)^4}{d}\Big)\bigg).
\end{equation}
}

Note that most terms on the right-hand side are independent of the entries of $M_s$. In fact, the only terms which depend on it are $\exp\Big(\frac{a\sqrt{d}}{1-p}\sum_{s<i<j\leq r} \by_i(s)\by_j(s)\Big)$. Thus, we focus on bounding their expectation.

\begin{claim}\label{claim:auxiliary_2}
If the entries of $M_s$ are sampled as Gaussians $\cN(0, \frac{1}{d})$, then we have
{\smaller[0.3]
\begin{align}
    \bE_{M_s}\bigg[\exp\bigg(\frac{a\sqrt{d}}{1\!-\!p}\!\!\sum_{s<i<j\leq r}\!\! \by_i(s)\by_j(s)\bigg)\bigg| \bigwedge_{s<i\leq r}\!\! E_{s, i}, M[s\!-\!1]\bigg]\leq  \exp\bigg(\frac{a^3}{(1\!-\!p)^3\sqrt{d}}\binom{r\!-\!s}{2}\!+\!O\Big(\frac{(r\!-\!s)^3}{d}\!+\!\frac{r-s}{\sqrt{d}}\Big)\bigg)\label{eq:auxiliary_2_blue}    \\
    \bE_{M_s}\bigg[\exp\bigg(\!-\!\frac{a\sqrt{d}}{p}\!\!\sum_{s<i<j\leq r}\!\! \by_i(s)\by_j(s)\bigg)\bigg| \bigwedge_{s<i\leq r}\!\! \overline{E}_{s, i}, M[s\!-\!1]\bigg]\leq \exp\bigg(\!-\!\frac{a^3}{p^3\sqrt{d}}\binom{r\!-\!s}{2}\!+\!O\Big(\frac{(r\!-\!s)^3}{d}\!+\!\frac{r-s}{\sqrt{d}}\Big)\bigg)\label{eq:auxiliary_2_red}
\end{align}
}
\end{claim}

To complete the induction step using the two auxiliary claims, we perform some simple but tedious bookkeeping. By plugging in the bound (\ref{eq:auxiliary_2_blue}) into (\ref{eq:induction_hypothesis_blue}), we obtain
{\smaller[0.3]
\begin{align*}
&\bE_{M_s}\Big[\Pb\big[C_{s}\wedge B_r\big| M_{s}\big]\Big| \bigwedge_{s<i\leq r}E_{s, i}, M[s\!-\!1]\Big]\leq \\
&\leq (1\!-\!p)^{\binom{r-s}{2}}\exp\bigg(\frac{a\sqrt{d}}{1\!-\!p}\!\!\sum_{s<i< j\leq r}\!\!\!\!\langle \pi_{s-1}(\by_i), \pi_{s-1}(\by_j)\rangle\!+\!\frac{a^3}{(1\!-\!p)^3\sqrt{d}}\binom{r\!-\!s}{3}\!+\!O\Big(\frac{(r\!-\!s)^4}{d}\!+\!\frac{(r\!-\!s)^2}{\sqrt{d}}\Big)\bigg) \\
&\hspace{5cm}\cdot \exp\bigg(\frac{a^3}{(1\!-\!p)^3\sqrt{d}}\binom{r\!-\!s}{2}\!+\!O\Big(\frac{ (r\!-\!s)^3}{d}\!+\!\frac{r-s}{\sqrt{d}}\Big)\bigg)\\
&\leq (1\!-\!p)^{\binom{r-s}{2}}\exp\bigg(\frac{a\sqrt{d}}{1\!-\!p}\!\!\sum_{s<i< j\leq r}\!\!\!\!\langle \pi_{s-1}(\by_i), \pi_{s-1}(\by_j)\rangle\!+\!\frac{a^3}{(1\!-\!p)^3\sqrt{d}}\binom{r\!-\!s\!+\!1}{3}\!+\!O\Big(\frac{(r\!-\!s\!+\!1)^4}{d}\!+\!\frac{(r\!-\!s\!+\!1)^2}{\sqrt{d}}\Big)\bigg).
\end{align*}
}

By plugging into the equation (\ref{eq:total_probability_blue}) the bound on $\Pb\Big[\bigwedge_{s<i\leq r}E_{s, i}\Big|M[s-1]\Big]$ from (\ref{eq:auxiliary_1_blue}) and the bound of (\ref{eq:induction_hypothesis_blue}), we obtain
{\smaller[0.3]
\begin{align*}
\Pb\Big[C_{s\!-\!1}\wedge B_r\Big| M[s\!-\!1]\Big]\leq    &(1\!-\!p)^{\binom{r-s}{2}}\exp\bigg(\frac{a\sqrt{d}}{1\!-\!p}\!\!\sum_{s<i< j\leq r}\!\!\!\!\langle \pi_{s-1}(\by_i), \pi_{s-1}(\by_j)\rangle\!+\!\frac{a^3}{(1\!-\!p)^3\sqrt{d}}\binom{r\!-\!s\!+\!1}{3}\\
&\hspace{7cm}\!+\!O\Big(\frac{(r\!-\!s\!+\!1)^4}{d}+\frac{(r\!-\!s\!+\!1)^2}{\sqrt{d}}\Big)\bigg)\\
&\hspace{3cm}\cdot(1\!-\!p)^{r-s}\exp\bigg(\frac{a\sqrt{d}}{1\!-\!p}\sum_{i=s+1}^r \langle \pi_{s-1}(\by_i), \pi_{s-1}(\by_s)\rangle\!+\!O\big((r\!-\!s)\delta\big)\bigg)\\
&\hspace{-4cm}\leq (1\!-\!p)^{\binom{r-s+1}{2}}\exp\bigg(\frac{a\sqrt{d}}{1\!-\!p}\!\!\sum_{s\leq i< j\leq r}\!\!\!\!\langle \pi_{s-1}(\by_i), \pi_{s-1}(\by_j)\rangle\!+\!\frac{a^3}{(1\!-\!p)^3\sqrt{d}}\binom{r\!-\!s\!+\!1}{3}\!+\!O\Big(\frac{(r\!-\!s\!+\!1)^4}{d}+\frac{(r\!-\!s\!+\!1)^2}{\sqrt{d}}\Big)\bigg).
\end{align*}
}
A quick inspection shows that this is precisely the bound we need on $\Pb\big[C_{s-1}\wedge B_r\big|M[s-1]\big]$, and thus our induction step is complete. Let us conclude by noting that nothing in this bookkeeping relied on the fact that we are estimating the probability of $C_s\wedge B_r$. In fact, an analogous proof would apply unchanged to estimate the probability of $I_s\wedge B_r$, and therefore we choose not to spell it out explicitly. Let us now finish the discussion of Proposition~\ref{prop:induction_1} by proving the auxiliary claims.

\begin{proof}[Proof of Claim~\ref{claim:auxiliary_1}.]
Since the proofs of both bounds are analogous, we only show the first one. Recall that $E_{s, i}$ holds whenever $\by_i(s)\by_s(s)\geq -c_p/\sqrt{d}-\langle \pi_{s-1}(\by_i), \pi_{s-1}(\by_s)\rangle$. Note that the only random variable in this equation is $\by_i(s)$, since $\by_s(s)$ was fixed and all of the other ones are determined by $M[s-1]$. Since the variables $\by_{s+1}(s), \dots, \by_{r}(s)$ are independent, so are the events $E_{s, s+1}, \dots, E_{s, r}$.

Thus, it is sufficient to compute the probabilities of the individual events $E_{s, i}$. Thus,
{\smaller[0.3]
\[\Pb\big[E_{s, i}|M[s\!-\!1]\big]=\Pb\bigg[\by_i(s)\geq -\frac{1}{\by_s(s)}\Big(c_p/\sqrt{d}\!+\!\langle \pi_{s-\!1}(\by_i), \pi_{s-\!1}(\by_s)\rangle\Big)\bigg| M[s\!-\!1]\bigg].\]
}
Recall that $\by_s(s)\in (1-2\delta, 1+\delta)$, and note that by Cauchy-Schwarz inequality
$|\langle \pi_{s-1}(\by_i), \pi_{s-1}(\by_s)\rangle|\leq \|\pi_{s-1}(\by_i)\|\cdot \|\pi_{s-1}(\by_s)\|\leq \frac{\alpha^2 \ell}{d}=O(\frac{1}{D\sqrt{d}})$. Therefore the cutoff for $E_{s, i}| M[s\!-\!1]$ to hold equals 
$$b_i=-c_p/\sqrt{d}-\langle \pi_{s-1}(\by_i), \pi_{s-1}(\by_s)\rangle+O(\delta/{\sqrt{d}}).$$ 

Since $\by_i(s)\sim \cN(0, \frac{1}{d})$, log-concavity of $\Phi$ (see (\ref{eqn:log_concavity})) implies that
{\smaller[0.3]
\[\Pb[\by_i(s)\geq b_i]=\Pb[\cN(0, 1)\leq -\sqrt{d}b_i]\leq \Pb[\cN(0, 1)\leq c_p]\exp\Big(\frac{\phi(c_p)}{\Phi(c_p)}(-\sqrt{d}b_i\!-\!c_p)\Big).\]
}
In our setup, we have $\Pb[\cN(0, 1)\leq c_p]=\Phi(c_p)=1-p$ and $\phi(c_p)=a$. Thus, we have
{\smaller[0.3]
\begin{align*}
&\Pb[E_{s, i}]=\Pb[\by_i(s)\geq b_i]\leq (1\!-\!p)\exp\Big(\frac{a\sqrt{d}}{1\!-\!p}\langle \pi_{s-1}(\by_i), \pi_{s-1}(\by_s)\rangle\!+\!O(\delta)\Big),
\end{align*}
}
where we have used that $-\sqrt{d}b_i-c_p=\sqrt{d}\langle \pi_{s-1}(\by_i), \pi_{s-1}(\by_s)\rangle+O(\delta)$. Multiplying over all $s< i\leq r$, we obtain
{\smaller[0.3]
\[\prod_{s<i\leq r}\Pb\Big[E_{s, i}\Big|M[s\!-\!1]\Big]\leq (1\!-\!p)^{r-s}\exp\bigg(\frac{a\sqrt{d}}{1\!-\!p}\sum_{i=s+1}^r \langle \pi_{s-1}(\by_i), \pi_{s-1}(\by_s)\rangle\!+\!O\big((r\!-\!s)\delta\big)\bigg),\]
}
which is sufficient to complete the proof due to the independence of the events $E_{s, i}$.
\end{proof}

\begin{proof}[Proof of Claim~\ref{claim:auxiliary_2}.]
Again, we will focus on proving (\ref{eq:auxiliary_2_blue}), since the proof of (\ref{eq:auxiliary_2_red}) is analogous. Recall that $E_{s, i}$ holds whenever $\by_i(s)\by_s(s)\geq -c_p/\sqrt{d}-\langle \pi_{s-1}(\by_i), \pi_{s-1}(\by_s)\rangle$. Since the diagonal entries $\by_s(s)$ are fixed, the events $E_{s, i}$ are independent and depend only on $\by_i(s)$. So, conditioning on their intersection leaves the variables $\by_{s+1}(s), \dots, \by_{r}(s)$ independent. Moreover, conditioning on the event $E_{s, i}$, each $\by_i(s)$ follows the distribution of a lower truncated Gaussian with the cutoff $b_i=(-c_p/\sqrt{d}-\langle \pi_{s-1}(\by_i), \pi_{s-1}(\by_s)\rangle)/\by_s(s)$.

As we have observed already, truncated Gaussian random variables have variance proxy at most $1/d$, and so we can apply Lemma~\ref{lemma:exponential_moment_special_case} with $\lambda=\frac{a\sqrt{d}}{1-p}$, $k=r-s$, random variables $\by_{s+1}(s), \dots, \by_{r}(s)$ and $S=\sum_{s<i<j\leq r}\by_i(s)\by_j(s)$ to obtain
{\smaller[0.3]
\[\bE[e^{\lambda S}]\leq \exp\bigg( \lambda \bE[S] \!+\! \frac{\lambda^2k^2}{d}\sum_{j=s+1}^r (\bE \by_j(s))^2 \!+\!\frac{4\lambda k}{d}\bigg).\]
}
Here, it is important to verify that $d\geq 4\lambda k$, which follows since $d=D^2\ell^2, \lambda=aD\ell/(1-p)$ and $k=C\ell$, reducing the inequality to $D\ell\geq aC\ell/(1-p)$, which holds since $D$ is sufficiently larger than $C$. Hence, in order to complete the proof, we only need to verify that
{\smaller[0.3]
\[\lambda \bE[S] \!+\! \frac{\lambda^2k^2}{d}\sum_{j=s+1}^r (\bE \by_j(s))^2 \!+\!\frac{4\lambda k}{d}\leq \frac{a^3}{(1\!-\!p)^3\sqrt{d}}\binom{r\!-\!s}{2}\!+\!O\Big(\frac{(r\!-\!s)^2\delta}{\sqrt{d}}\!+\!\frac{(r\!-\!s)^3}{d}\!+\!\frac{r-s}{\sqrt{d}}\Big).\]
}
The last term on the left hand side is bounded by $4\lambda k/d\leq O(\sqrt{d}(r-s)/d)$.

Using the independence of $\by_{s+1}(s), \dots, \by_r(s)$, we can upper bound $\bE[S]$ as follows
{\smaller[0.3]
\[\bE[S]=\sum_{s<i<j\leq r}\bE[\by_i(s)]\bE[\by_j(s)]\leq \frac{1}{2}\Big(\sum_{s<i\leq r}\bE[\by_i(s)]\Big)^2\leq \frac{r\!-\!s}{2}\sum_{s<i\leq r}\bE[\by_i(s)]^2,\]
}
where we have used the inequality between the arithmetic and the quadratic mean in the last step. Thus, we have
{\smaller[0.3]
\[\lambda \bE[S] \!+\! \frac{\lambda^2k^2}{d}\sum_{j=s+1}^r (\bE \by_j(s))^2\leq \lambda \frac{r\!-\!s}{2}\Big(1\!+\!\frac{2\lambda k}{d}\Big)\sum_{i=s+1}^r \bE[\by_i(s)]^2.\]
}

We now need to calculate the expectations of $\by_i(s)$, which we do using Lemma~\ref{lemma:truncated_expectation}. Let us recall that $\by_i(s)$ is a lower truncated Gaussian with the cutoff $b_i=-(c_p/\sqrt{d}+\langle \pi_{s-1}(\by_i), \pi_{s-1}(\by_s)\rangle)/\by_s(s)$. If we write $b_i=-c_p/\sqrt{d}+\eps_i$, then $|\eps_i|=|b_i+c_p/\sqrt{d}|=O(\frac{1}{D\sqrt{d}}+\frac{\delta}{\sqrt{d}})=O(\frac{1}{D\sqrt{d}})$. The last inequality holds since $|\langle \pi_{s-1}(\by_i), \pi_{s-1}(\by_j)\rangle|\leq \frac{\alpha^2\ell}{d}\leq O(\frac{1}{D\sqrt{d}})$ and since $\by_s(s)$ is fixed and in $(1-2\delta, 1+\delta)$. Then, Lemma~\ref{lemma:truncated_expectation} gives
{\smaller[0.3]
\[\bE[\by_i(s)|E_{s, i}]=\frac{\phi(c_p)}{\Phi(c_p)\sqrt{d}}\!+\!O(|\eps_i|)=\frac{a}{(1\!-\!p)\sqrt{d}}\!+\!O\Big(\frac{1}{D\sqrt{d}}\Big).\]
}
Recalling that $\lambda=\frac{a\sqrt{d}}{1-p}$ and $\frac{r-s}{\sqrt{d}}\leq \frac{C\ell}{D\ell}\leq O(\frac{1}{D})$, we get
{\smaller[0.3]
\begin{align*}
\lambda \frac{r\!-\!s}{2} \Big(1\!+\!\frac{2\lambda (r\!-\!s)}{d}\Big)\sum_{i=s+1}^r \bE[\by_i(s)]^2&\leq \frac{a\sqrt{d}(r\!-\!s)}{2(1\!-\!p)}\Big(1\!+\!O\Big(\frac{r\!-\!s}{\sqrt{d}}\Big)\Big)\sum_{i=s+1}^r \left[\frac{a^2}{(1\!-\!p)^2d}\!+\!O\left(\frac{1}{Dd}\right)\right]\\
&\leq \frac{a^3}{(1\!-\!p)^3\sqrt{d}}\binom{r\!-\!s}{2}\Big(1\!+\!O\big(\frac{1}{D}\big)\Big),
\end{align*}
}
which is what we needed to show.
\end{proof}
\end{proof}

The only remaining step is to prove Proposition~\ref{prop:perfect_probabilities} from Proposition~\ref{prop:induction_1}.

\begin{proof}[Proof of Proposition~\ref{prop:perfect_probabilities}.]
By plugging in $s=0$ in Proposition~\ref{prop:induction_1}, we find that
{\smaller[0.3]
\begin{align*}
    \Pb[\bx_1, \dots, \bx_r \text{ form a red clique and are perfect}]&\leq p^{\binom{r}{2}}\exp\bigg(\!-\!\frac{a^3}{p^3\sqrt{d}}\binom{r}{3}\!+\!O\Big(\frac{r^4}{d}\!+\!\frac{r}{\sqrt{d}}\Big)\bigg)\\
    &=p^{\binom{r}{2}}\exp\bigg(\!-\!\frac{a^3}{p^3\sqrt{d}}\binom{r}{3}\Big(1\!+\!O\Big(\frac{1}{D}\Big)\Big)\bigg),
\end{align*}
}
since $\frac{r}{\sqrt{d}}\leq O(1/D)$. Similarly,
{\smaller[0.3]
\begin{align*}
    \Pb[\bx_1, \dots, \bx_r \text{ form a blue clique and are perfect}]&\leq (1\!-\!p)^{\binom{r}{2}}\exp\bigg(\frac{a^3}{(1\!-\!p)^3\sqrt{d}}\binom{r}{3}\!+\!O\Big(\frac{r^4}{d}\!+\!\frac{r}{\sqrt{d}}\Big)\bigg)\\
    &=(1\!-\!p)^{\binom{r}{2}}\exp\bigg(\frac{a^3}{(1\!-\!p)^3\sqrt{d}}\binom{r}{3}\Big(1\!+\!O\Big(\frac{1}{D}\Big)\Big)\bigg).
\end{align*}
}
This completes the proof.
\end{proof}

\section{Concluding remarks}\label{sec:concluding}

In this paper, we have presented an improvement upon the lower bounds for Ramsey numbers $R(\ell, C\ell)$ using so-called Gaussian random graphs. The simplified analysis we develop yields in particular $R(\ell, C\ell)\geq ((1+\eps)p_C^{-1/2})^{\ell}$ for some absolute constant $\eps>0$, at least when $C$ is bounded away from $1$. However, we suspect that the same construction may in fact yield even the stronger lower bound $R(\ell, C\ell)\geq p_C^{-(1+\eps)\ell/2}$ for some absolute constant $\eps>0$, when $C$ is bounded away from $1$.

Let us briefly and informally outline the intuition behind this belief, as well as the steps we think are necessary to get there. In Proposition~\ref{prop:clique_probabilities}, we showed that
{\smaller[0.3]
\begin{align*}
&P_{\red, \ell}\lesssim p^{\binom{\ell}{2}}\exp\bigg(-\frac{a^3}{p^3\sqrt{d}}\binom{\ell}{3}\bigg)\lesssim p^{\binom{\ell}{2}}\exp\bigg(-\frac{a^3}{3p^3 D}\binom{\ell}{2}\bigg),
\end{align*}
}\noindent
where we used $d=D^2\ell^2$. This estimate shows that the smaller $D$ is, the better the bound on $P_{\red, \ell}$. Thus, the key question is the following: \textit{how small can $D$ be while still ensuring that our bounds on $P_{\red, \ell}$ and $P_{\blue, C\ell}$ hold?}

In Appendix~\ref{sec:appendix_2}, we showed that $D$ can be taken as small as $\Theta((\log p^{-1})^{3/2})$. Since $a=\phi(\Phi^{-1}(p))=(1+o(1))p\sqrt{2\log p^{-1}}$ (where $o(1)\to 0$ as $p\to 0$), we roughly obtain
{\smaller[0.3]
\begin{align*}
&P_{\red, \ell}\lesssim \Big(p\exp\big(-\frac{a^3}{3p^3D}\big)\Big)^{\binom{\ell}{2}}\approx \bigg(p\exp\Big(-\frac{(2\log p^{-1})^{3/2}}{3D}\Big)\bigg)^{\binom{\ell}{2}}\approx \big(p(1-\eps)\big)^{\binom{\ell}{2}}, \text{ for some fixed }\eps>0.
\end{align*}
}\noindent
Since a similar bound can be established for $P_{\blue, C\ell}$, this ultimately yields $R(\ell, C\ell)\geq ((1+\eps)p_C^{-1/2})^{\ell}$. However, if one could show that $D$ can be taken as low as $\Theta(\sqrt{\log p^{-1}})$, then we would have
{\smaller[0.3]
\begin{align*}
&P_{\red, \ell}\lesssim \bigg(p\exp\Big(-\frac{(2\log p^{-1})^{3/2}}{3D}\Big)\bigg)^{\binom{\ell}{2}}\approx \bigg(p\exp\Big(-\eps\log p^{-1}\Big)\bigg)^{\binom{\ell}{2}}\approx \big(p^{1+\eps}\big)^{\binom{\ell}{2}}, \text{ for some fixed }\eps>0.
\end{align*}
}\noindent
If analogous bounds held for $P_{\blue, C\ell}$ as well, then this would immediately give $R(\ell, C\ell)\geq p_C^{-(1+\eps)\ell/2}$. Moreover, note that Proposition~\ref{prop:perfect_probabilities_appendix_red} \textit{already establishes} a bound of this type for $P_{\red, \ell}$ when $D=\Theta(\sqrt{\log p^{-1}})$. Thus, to obtain the bound $R(\ell, C\ell)\geq p_C^{-(1+\eps)\ell/2}$, it would suffice to improve the analysis of blue cliques so as to allow $D=\Theta(\sqrt{\log p^{-1}})$. At present, Proposition~\ref{prop:perfect_probabilities_appendix_blue} requires $D=\Theta((\log p^{-1})^{3/2})$, which is too large for this improvement, as discussed above.

Finally, if one can establish $R(\ell, C\ell)\geq p_C^{-(1+\eps)\ell/2}$, then further improvements to lower bounds for off-diagonal Ramsey numbers $R(s,t)$, with $s$ fixed and $t\to\infty$, may be within reach.

\vspace{0.3cm}
\noindent
{\bf Acknowledgments.} We would like to thank Kiril Bangachev, Bo'az Klartag and Petar Nizi\'c-Nikolac for insightful discussions about this project.

\appendix

\section{Proofs of Preliminaries}\label{sec:appendix_1}

\begin{proof}[Proof of Lemma~\ref{lemma:norm_concentration}.]
If $\|\bx\|_2\notin (1-\delta, 1+\delta)$, then we also have $\|\bx\|_2^2\notin (1-\delta, 1+\delta)$. As we have already observed, if we define $Y = d\|\bx\|_2^2$ we have $Y\sim\chi^2_d$.
Thus, we can write
{\smaller[0.3]
\[\Pb\big[\|\bx\|_2\notin (1-\delta, 1+\delta)\big]\leq \Pb\big[\|\bx\|_2^2\notin (1-\delta, 1+\delta)\big]=\Pb\Big[Y\notin \big((1-\delta)d, (1+\delta)d\big)\Big].\]
}

By Laurent-Massart inequality we have $\Pb[|Y-d|\geq 2\sqrt{dt}+2t]\leq 2e^{-t}$,
and plugging in $t=\delta^2d/10$, we obtain
{\smaller[0.3]
\[\Pb\big[\|\bx\|_2\notin (1-\delta, 1+\delta)\big]= \Pb\Big[|Y-d|\geq \delta d\Big]\leq \Pb\Big[|Y-d|\geq 2\sqrt{dt}+2t\Big]\leq 2e^{-t}=2\exp(-\delta^2 d/10),\]
}
where we have used that $2\sqrt{dt}+2t=2\frac{\delta d}{\sqrt{10}}+2\frac{\delta^2 d}{10}\leq \delta d$ since $\delta<1$.
\end{proof}

\begin{proof}[Proof of Lemma~\ref{lemma:norm_projection}.]
Since the normal distribution $\cN(0, \frac{1}{d}I_d)$ is invariant under the action of the orthogonal group, i.e. under rotations, we may assume that $W$ is simply the subspace spanned by the first $s$ coordinate vectors $\be_1, \dots, \be_s$.

It is then easy to see that $\pi_W(\bx)$ follows the distribution $\cN(0, \frac{1}{d}I_s)$ inside $W$. Hence, $\sqrt{d}\pi_W(\bx)\sim \cN(0, I_{s})$ and we can apply Theorem~\ref{thm:laurent_massart} to obtain a tail bound on $\|\sqrt{d}\pi_W(\bx)\|_2^2\sim \chi_s^2$. In particular, we have
{\smaller[0.3]
\[\Pb\Big[\|\pi_W(\bx)\|_2\geq \frac{\alpha\sqrt{\ell}}{\sqrt{d}}\,\,\Big]= \Pb\bigg[\,\Big\|\sqrt{d}\pi_W(\bx)\Big\|_2\geq \alpha\sqrt{\ell}\,\bigg]\leq \Pb\bigg[\chi_s^2\geq \alpha^2 \ell\bigg]\leq \Pb\bigg[\chi_s^2>s+\frac{3}{4}\alpha^2\ell\bigg].\]
}
Picking $t=\frac{\alpha^2 \ell}{8}$, we get $2\sqrt{st}+2t = 2\sqrt{\frac{\alpha^2 \ell s}{8}}+2\frac{\alpha^2 \ell}{8}\leq \frac{3\alpha^2 \ell}{4}$, where the inequality follows since $\alpha^2 \ell/8>s$. Hence, by the Laurent-Massart inequality, we get
{\smaller[0.3]
\[\Pb\Big[\|\pi_W(\bx)\|_2\geq \frac{\alpha\sqrt{\ell}}{\sqrt{d}}\,\,\Big]\leq \Pb\bigg[\chi_s^2>s+2\sqrt{st}+2t\bigg]\leq \exp(-\alpha^2\ell/8).\qedhere\]
}
\end{proof}

\begin{proof}[Proof of Lemma~\ref{lemma:truncated_expectation}.]
We will start by proving the exact formulas, when the cutoff is exactly $b$. Let $Z=\sqrt{d}X$ be the standard Gaussian. Observe that the function $\phi(z)$ satisfies $\phi'(z) = -ze^{-z^2/2}/\sqrt{2\pi}=-z\phi(z)$. Hence, for any $t\in \bR$, we can write $\int_{-\infty}^{t} z \phi(z) \, dz =  -\phi(t)$. We can now compute the expectations of the truncated Gaussians:
{\smaller[0.3]
\begin{align*}
    \bE[X | X \leq b/\sqrt{d}] = \frac{1}{\sqrt{d}}\bE\bigg[Z \,\Big|\, Z \leq b\bigg]
    = \frac{1}{\sqrt{d}} \cdot \frac{\int_{-\infty}^{b} z \phi(z) \, dz}{\Pb[Z \leq b]}
    = \frac{1}{\sqrt{d}} \cdot \frac{-\phi(b)}{\Phi(b)}.
\end{align*}
}
To compute the other expectation, we can symmetrically write $\int_{t}^{\infty} z \phi(z) \, dz = \phi(t)$. Then,
{\smaller[0.3]
\begin{align*}
    \bE[X | X \geq b/\sqrt{d}] =\frac{1}{\sqrt{d}}\bE\bigg[Z \,\Big|\, Z \geq b\bigg] =  \frac{1}{\sqrt{d}} \cdot \frac{\int_{b}^{\infty} z \phi(z) \, dz}{\Pb[Z \geq b]} = \frac{1}{\sqrt{d}} \cdot \frac{\phi(b)}{1 - \Phi(b)}.
\end{align*}
}

To prove the second part of the statement, we define the function $f(t) = \bE[X | X \leq t/\sqrt{d}]$. We have already shown that $f(t)=-\frac{\phi(t)}{\Phi(t)\sqrt{d}}$. Moreover, note that $f'(t)=\frac{t\phi(t)\Phi(t)+\phi(t)^2}{\Phi(t)^2\sqrt{d}}$ is a function bounded by $O(1/\sqrt{d})$ on the real line. This is because $\sqrt{d}f'(t)=\frac{\phi(t)}{\Phi(t)}\big(t+\frac{\phi(t)}{\Phi(t)}\big)$ is a smooth function on the real line with $\lim_{t\to \infty}\sqrt{d}f'(t)=0$ and
{\smaller[0.3]
\[2\geq \Big(|t|+\frac{1}{|t|}\Big)\frac{1}{|t|}\geq \frac{\phi(t)}{\Phi(t)}\Big(t+\frac{\phi(t)}{\Phi(t)}\Big)\geq 0,\]
}
for all $t\leq -1$ due to (\ref{eqn:mills_ratio}). Finally, by the Mean value theorem, we conclude that $f(b+\eps) = f(b) + f'(c) \cdot \eps$ for some $c\in(b, b+\eps)$. From this, we have
{\smaller[0.3]
\[\bE\bigg[X | X \leq \frac{b+\eps}{\sqrt{d}}\bigg] = \bE\bigg[X | X \leq \frac{b}{\sqrt{d}}\bigg] + f'(c)\eps = -\frac{\phi(b)}{\Phi(b)\sqrt{d}} + O\bigg(\frac{\eps}{\sqrt{d}}\bigg), \]
}
which is what we needed. The computation of $\bE[X | X \geq (b+\eps)/\sqrt{d}]$ is analogous.
\end{proof}

\begin{proof}[Proof of Lemma~\ref{lemma:exponential_square_moments}.]
Before we start with the proof, observe that $X$ satisfies Gaussian-like tail bounds. Namely, for any $u\in (0, \infty)$, we can set $\theta=u/\sigma^2$ and obtain the following
{\smaller[0.3]
\[\Pb[X\geq u]=\Pb[e^{\theta X}\geq e^{\theta u}]\leq \frac{\bE[e^{\theta X}]}{e^{\theta u}}\leq e^{\theta^2 \sigma^2/2-\theta u}= e^{-u^2/2\sigma^2}.\]
}
Note that the first inequality here is due to Markov's inequality and the second one to the subgaussian property of $X$. Thus, we also have $\Pb[|X|\geq u]\leq 2e^{-u^2/2\sigma^2}$

Now, we can compute the needed expectation as follows. Since $e^{\lambda X^2}\geq 1$, we have
{\smaller[0.3]
\[\bE[e^{\lambda X^2}]=1+\int_1^\infty \Pb[e^{\lambda X^2}\geq t]dt=1+\int_1^\infty \Pb\bigg[|X|\geq \sqrt{\frac{\log t}{\lambda}}\,\,\bigg]dt\leq 1+2\int_{1}^\infty e^{-\log t/2\lambda \sigma^2}dt,\]
}
where the last inequality comes from the tail bound on $X$. Finally, the integral in question is easily computable when $2\lambda\sigma^2<1$ since then we have
{\smaller[0.3]
\[\bE[e^{\lambda X^2}]\leq 1+2\int_{1}^\infty t^{-1/2\lambda \sigma^2}dt=1+2\bigg(\frac{t^{1-1/2\lambda\sigma^2}}{1-1/2\lambda\sigma^2}\bigg|_{1}^\infty\bigg)=1-2\frac{1}{1-1/2\lambda\sigma^2}=1+\frac{4\lambda \sigma^2}{1-2\lambda\sigma^2}.\qedhere\]
}
\end{proof}

\begin{proof}[Proof of Lemma~\ref{lemma:exponential_moment_special_case}.]
If $A$ is the all-ones matrix with zeros on the diagonal, then $S=\frac{1}{2}X^\top A X$, where $X=(X_1, \dots, X_k)^\top$ is a subgaussian vector in $\bR^k$. Then, we can write
{\smaller[0.3]
\[X^\top AX= \bE[X]^\top A \bE[X]+2 (X -\bE[X])^\top A \bE[X] + (X-\bE X)^\top A (X-\bE X).\]
}
Introducing $Y=X-\bE X$ and using that $A$ is zero on the diagonal, this reduces to $X^\top AX=\bE[X^\top A X]+2 Y^\top A \bE[X] + Y^\top A Y$. If we let $\Sigma_0=\frac{1}{2}\bE[X^\top A X]=\bE[S], \Sigma_1 = Y^\top A \bE[X]$ and $\Sigma_2=\frac{1}{2}Y^\top A Y$, then we can write $S=\Sigma_0+\Sigma_1+\Sigma_2$, and thus  $\bE\big[e^{\lambda S}\big]=e^{\lambda\Sigma_0}\bE[e^{\lambda\Sigma_1}e^{\lambda\Sigma_2}]$. The Cauchy-Schwarz inequality then gives
{\smaller[0.3]
\begin{equation}\label{eqn:sigma_decomposition}
\bE\big[e^{\lambda S}\big] \leq e^{\lambda\Sigma_0}\sqrt{\bE[e^{2\lambda\Sigma_1}]}\sqrt{\bE[e^{2\lambda\Sigma_2}]}.
\end{equation}
}
Since the entries of $Y$ are independent, we have
{\smaller[0.3]
\begin{align}
\bE[e^{2\lambda\Sigma_1}]&\leq \prod_{i=1}^k \bE\big[e^{2\lambda (A\bE[X])_i\cdot Y_i}\big]\leq \prod_{i=1}^k \exp\bigg(\frac{2\lambda^2 (A\bE[X])_i^2}{d}\bigg)\notag\\
&\leq \exp\bigg(\sum_{i=1}^k \frac{2\lambda^2}{d}k\sum_{j=1}^k (\bE X_j)^2\bigg)=\exp\bigg(\frac{2\lambda^2k^2}{d}\sum_{j=1}^k (\bE X_j)^2\bigg),\label{eqn:sigma_1_bound}
\end{align}
}
where the second inequality follows since $Y_i$ is subgaussian with mean zero and variance proxy at most $1/d$ and the third inequality by Cauchy-Schwarz.

Thus, in order to complete the proof, we only need to show that $\bE[e^{2\lambda \Sigma_2}]\leq \exp(8|\lambda| k/d)$. To do this, we consider two cases - when $\lambda\geq 0$ and when $\lambda<0$. Let us start with $\lambda\geq 0$. Since $2\Sigma_2=\sum_{i\neq j}Y_iY_j=\Big(\sum_{i=1}^k Y_i\Big)^2-\sum_{i=1}^k Y_i^2\leq \Big(\sum_{i=1}^k Y_i\Big)^2$, we have
{\smaller[0.3]
\[\bE\big[e^{2\lambda\Sigma_2}\big]\leq \bE\bigg[\exp\bigg(\lambda\bigg(\sum_{i=1}^k Y_i\bigg)^2\bigg)\bigg].\]
}
Note that the random variable $Z=\sum_{i=1}^k Y_i$ is a subgaussian random variable with variance proxy at most $k/d$, as a sum of $k$ subgaussians of variance proxy $1/d$. Therefore, using Lemma~\ref{lemma:exponential_square_moments} we can bound the expectation of $\exp(\lambda Z^2)$ as
{\smaller[0.3]
\begin{equation*}
    \bE\big[e^{2\lambda\Sigma_2}\big]\leq \bE\big[e^{\lambda Z^2}\big]\leq 1+\frac{4\lambda k/d}{1-2\lambda k/d}\leq 1+\frac{8\lambda k}{d}\leq \exp\bigg(\frac{8\lambda k}{d}\bigg).
\end{equation*}
}
When $\lambda<0$, we use $2\Sigma_2=\Big(\sum_{i=1}^k Y_i\Big)^2-\sum_{i=1}^k Y_i^2\geq -\sum_{i=1}^k Y_i^2$, leading to
{\smaller[0.3]
\[\bE\big[e^{2\lambda\Sigma_2}\big]\leq \bE\bigg[\exp\bigg(-\lambda\sum_{i=1}^k Y_i^2\bigg)\bigg]=\prod_{i=1}^k \bE[e^{-\lambda Y_i^2}],\]
}
where the last inequality is due to the independence of the variables $Y_i$. Since $Y_i$ has variance proxy at most $1/d$, by Lemma~\ref{lemma:exponential_square_moments}, we find that
{\smaller[0.3]
\[\bE[e^{2\lambda\Sigma_2}]\leq \prod_{i=1}^k \bE[e^{-\lambda Y_i^2}]\leq \prod_{i=1}^k \bigg(1+\frac{4|\lambda|/d}{1-2|\lambda|/d}\bigg)\leq \prod_{i=1}^k \exp\bigg(\frac{8|\lambda|}{d}\bigg)=\exp\bigg(\frac{8|\lambda| k}{d}\bigg).\]
}
By plugging in the bounds (\ref{eqn:sigma_1_bound}) and $\bE[e^{2\lambda \Sigma_2}]\leq \exp(8|\lambda| k/d)$ into (\ref{eqn:sigma_decomposition}), we find that
{\smaller[0.3]
\[\bE\big[e^{\lambda S}\big]\leq e^{\lambda \bE[S]}\sqrt{\exp\bigg(\frac{2\lambda^2k^2}{d}\sum_{j=1}^k (\bE X_j)^2\bigg)}\sqrt{\exp\bigg(\frac{8|\lambda| k}{d}\bigg)}\leq \exp\bigg( \lambda \bE[S] + \frac{\lambda^2k^2}{d}\sum_{j=1}^k (\bE X_j)^2 +\frac{4|\lambda| k}{d}\bigg),\]
}
which is what we needed to prove.
\end{proof}

\section{Tightening the bounds}\label{sec:appendix_2}

The purpose of this appendix is to present an optimized version of the method from Section~\ref{sec:induction}, extending it to the regime where $D$ is not much larger than $C$. As indicated in Section~\ref{sec:concluding}, this allows for improved quantitative bounds on $R(\ell, C\ell)$, which we present in the following theorem.

\begin{theorem}\label{thm:main_appendix}
For every sufficiently large constant $C$, there is an integer $\ell_0(C)$ such that for all $\ell>\ell_0(C)$ we have 
{\smaller[0.3]
\[R(\ell, C\ell)\geq (e^{1/24}p_C^{-1/2})^\ell,\]
}
where $p_C\in (0, 1/2)$ is the unique solution to the equation $C=\frac{\log p_C}{\log (1-p_C)}$.
\end{theorem}
In the proof of this theorem, we first take $C$ sufficiently large, define $p_C$, choose $p=p_C+1/C$ and $D=4aC/(1-p)$ in the final bookkeeping step, and then take $\ell$ sufficiently large as a function of $C$.

Unlike Section~\ref{sec:induction} where we had a lot of room to spare, in the proof of Theorem~\ref{thm:main_appendix} we must be more careful with various estimates. As a consequence, we prefer to present the arguments for bounding the probability of blue and red cliques separately, since they are optimized in slightly different ways. We hope that this will contribute to the clarity, even though certain parts of the argument from Section~\ref{sec:induction} will be repeated twice in what follows.

For example, in Section~\ref{sec:perfect}, we said that a sequence $\bx_1, \dots, \bx_r$ is perfect if for all $1\leq i \leq r$ we have:
{\smaller[0.3]
\[\|\bx_i\|_2\in (1-\delta, 1+\delta),\text{ and } \|\pi_{\spn\{\bx_1, \dots, \bx_{i-1}\}}(\bx_i)\|_2 \leq \frac{\alpha\sqrt{\ell}}{\sqrt{d}},\]
}
where $\alpha=10\sqrt{C\log (10/p)}$ and $\delta=\alpha d^{-1/4}$. The main reason for choosing this value of $\alpha$ was to ensure that, given $\bx_1, \dots, \bx_{i-1}$, the probability that a random vector $\bx_i$ violates the second condition is at most $(p/10)^{10C\ell}$ (see the proof of Proposition~\ref{prop:clique_probabilities}). However, observe that this is only necessary when dealing with blue cliques, which have up to $C\ell$ vertices: if the probability that a random vector $\bx_i$ violates the second condition is at most $(p/10)^{10\ell}$, the derivation of the Proposition~\ref{prop:clique_probabilities} from the bounds on $P_{\red, r}^*$ would go through unchanged.  

This means that one could use a different value of $\alpha$ for red and blue cliques, and this is precisely what we will do in this appendix by setting $\alpha_{\blue} = 10\sqrt{C\log (10/p)}$ as before, but setting $\alpha_{\red} = 10\sqrt{\log (10/p)}$. Since $C\approx p_C^{-1}\log p_C^{-1}$, this saves a factor of approximately $p_C^{-1/2}$ in $\alpha_{\red}$, which will be very useful to us.

To reiterate, in this appendix, we will say a sequence $\bx_1, \dots, \bx_r$ with $1\leq r\leq \ell$ is \textit{red-perfect} if for all $1\leq i \leq r$ we have 
{\smaller[0.3]
\[\|\bx_i\|_2\in (1-\delta, 1+\delta),\text{ and } \|\pi_{\spn\{\bx_1, \dots, \bx_{i-1}\}}(\bx_i)\|_2 \leq \frac{\alpha_{\red}\sqrt{\ell}}{\sqrt{d}},\]
}
and we will say that a sequence $\bx_1, \dots, \bx_r$ with $1\leq r\leq C\ell$ is \textit{blue-perfect} if for all $1\leq i \leq r$ we have
{\smaller[0.3]
\[\|\bx_i\|_2\in (1-\delta, 1+\delta),\text{ and } \|\pi_{\spn\{\bx_1, \dots, \bx_{i-1}\}}(\bx_i)\|_2 \leq \frac{\alpha_{\blue}\sqrt{\ell}}{\sqrt{d}},\]
}
where $\delta=\alpha_{\blue} d^{-1/4}$ in both cases (as before, the particular value of $\delta$ is not of crucial importance). Naturally, we denote by $P_{\red, r}^*$ the probability that the sequence $x_1, \dots, x_r$ is red-perfect and forms a red clique, and similarly for $P_{\blue, r}^*$.

Throughout this appendix, we take over the notation and definitions from the main body of the paper. Additionally, we change the convention, which was used in the main body, that the implied constant in $O(\cdot)$ can depend on $C$ and $p$. In this appendix, the implied constant $O(\cdot)$ will always be absolute.

\subsection{Blue cliques} We begin by showing an upper bound on $P_{\blue, r}^*$, which corresponds to the bound shown in Proposition~\ref{prop:clique_probabilities}.

\begin{proposition}\label{prop:perfect_probabilities_appendix_blue}
Let $C>1$, let $p$ be a sufficiently small real number, and let $\ell, d=D^2\ell^2$ be large positive integers, with $D\geq \frac{4aC}{1-p}$. Then, for every integer $10\leq r\leq C\ell$, we have
{\smaller[0.3]
\begin{align}
P_{\blue, r}^*\leq (1-p)^{\binom{r}{2}}\exp\bigg(\frac{4a^3}{(1-p)^3\sqrt d}\binom{r}{3} + d^{-1/5}\binom{r}{2}+ r\bigg),\label{eq:perfect_probabilities_appendix_blue}
\end{align}
}
where $a=\phi(c_p)$.
\end{proposition}

Note that the terms containing $\binom{r}{2}$ and $r$ are lower-order terms when $r\approx \sqrt{d}$. We present them only to ensure that the statement still holds for all $r\geq 3$, even though we will only apply it with $r=C\ell$. 
The function 
{\smaller[0.3]
\[
f_r(x)=\frac{ar}{(1-p)\sqrt{d}}\cdot \frac{\phi(x)^2}{\Phi(x)^2}+\log\Phi(x)
\]
}
will play an important role in the proof, for technical reasons which we will explain later. This function has several convenient properties - it is smooth, negative, increasing and concave whenever $\kappa=\frac{(1-p)\sqrt d}{ar}\geq 4$ (which follows from the assumption $D\geq 4aC/(1-p)$). Let us remark that the exact constant $\frac{ar}{(1-p)\sqrt d}$ in definition of the function $f_r$ is not of crucial importance - the constant could have easily been replaced with say $\frac{1}{4}$ and the proof would go through essentially unchanged. However, we choose this particular normalization to highlight the similarity between Proposition~\ref{prop:perfect_probabilities_appendix_blue} and Proposition~\ref{prop:perfect_probabilities}, at the expense of having slightly longer formulas.

\begin{lemma}\label{lemma:function_analysis}
For every $\kappa\geq 4$, the function
{\smaller[0.3]
\[
f(x)=\frac{\phi(x)^2}{\kappa\Phi(x)^2}+\log\Phi(x)
\]
}
is negative, increasing, and concave for $x\in \bR$.
\end{lemma}
\begin{proof}
To simplify the notation, we will define $R(x) = \phi(x)/\Phi(x)$. It is easy to compute the derivative of $R$, which is $R'(x) = \frac{\phi'(x)\Phi(x)-\phi(x)\Phi'(x)}{\Phi(x)^2}= -x R(x) - R(x)^2$, since $\Phi'(x)=\phi(x)$ and $\phi'(x)=-x\phi(x)$. Also, recall that a standard inequality of Sampford implies that $R'(x)\in (-1, 0)$ for all $x\in \bR$ (see equation (3) in \cite{Sampford}). 

First consider $g(x)=\frac{1}{4}R(x)^2+\log \Phi(x)$. Differentiating, we obtain $g'(x) = \frac{1}{2}R(x)R'(x) + \frac{\phi(x)}{\Phi(x)} = R(x)\left[1 + \frac{1}{2}R'(x)\right]$. Since $R'(x)\in (-1, 0)$, we have $1 + \frac{1}{2}R'(x) > 1/2$, and since $R(x) > 0$, we have $g'(x) > 0$. 
As $g$ is strictly increasing and $\lim_{x \to \infty} g(x) = 0$, it follows that $g(x) < 0$ for all finite $x$.

To establish concavity, we compute the second derivative:
{\smaller[0.3]
\[
g''(x) = \frac{1}{2} \left[ (R'(x))^2 + R(x)R''(x) \right] + R'(x).
\]
}
Differentiating the defining equation for $R'(x)$ yields $R''(x) = -R(x) - x R'(x) - 2R(x)R'(x)$. Substituting this into the expression for $g''(x)$ and simplifying by utilizing the relation $-x R(x) = R'(x) + R(x)^2$, we arrive at the factorization:
{\smaller[0.3]
\[
2g''(x) = (1 + R'(x))(2R'(x) - R(x)^2).
\]
}
Since $R'(x) > -1$, the term $(1+R'(x))$ is strictly positive. Conversely, as $R'(x) < 0$ and $R(x)^2 > 0$, the term $(2R'(x) - R(x)^2)$ is strictly negative. Therefore, $g''(x) < 0$ for all $x \in \mathbb{R}$, proving concavity.

For $\kappa\geq 4$, we can write
{\smaller[0.3]
\[
f(x)=\frac{4}{\kappa}g(x)+\left(1-\frac{4}{\kappa}\right)\log\Phi(x).
\]
}
The function $\log\Phi(x)$ is negative, increasing, and concave, and the coefficients in the above linear combination are nonnegative. Hence $f$ is also negative, increasing, and concave.
\end{proof}

We now focus on proving Proposition~\ref{prop:perfect_probabilities_appendix_blue}, by an induction argument similar to the proof of Proposition~\ref{prop:induction_1}. We now give our induction statement in the form of Proposition~\ref{prop:induction_appendix_blue}. Recall that $C_s$ denotes the event that the vertices $s+1, \dots, r$ form a clique, and that $B^{\blue}_r$ is the event that $\by_1, \dots, \by_r$ form a perfect sequence. Note that setting $s=0$ in the statement of Proposition~\ref{prop:induction_appendix_blue} gives precisely the statement of Proposition~\ref{prop:perfect_probabilities_appendix_blue}, since $r\cdot \binom{r}{2}\leq 4\binom{r}{3}$.

\begin{proposition}\label{prop:induction_appendix_blue}
Under the assumptions of Proposition~\ref{prop:perfect_probabilities_appendix_blue}, let $s, r$ be nonnegative integers such that $1\leq s+1\leq r\leq C\ell$. Suppose that the first $s$ columns of $M$ are fixed, and denote them by $M[s]$. Then
{\smaller[0.3]
\begin{align}
    \mathbb{P}\Big[C_s\wedge B^{\blue}_r\Big| M[s]\Big]\leq (1\!-\!p)^{\binom{r-s}{2}}\exp\bigg(&\mu\sqrt{d}\!\!\!\!\sum_{s<i< j\leq r}\!\!\!\!\langle \pi_s(\by_i), \pi_s(\by_j)\rangle \!+\! \left(\frac{a^3r}{(1\!-\!p)^3\sqrt{d}}\!+\!d^{-\frac{1}{5}}\right)\binom{r\!-\!s}{2} \!+\! (r\!-\!s)\!\bigg), \label{eq:induction_bound_blue_2}
\end{align}
}
where $a=\phi(c_p)$ and $\mu=f_r'(c_p)$.
\end{proposition}
\begin{proof}
We argue by reverse induction on $s$. The base case $s=r$ is trivial. Thus, let us show how to do a step of the induction, from $s$ to $s-1$. We are given the columns $M_1, \dots, M_{s-1}$, and we sample the column $M_s$ by setting all of its entries below the diagonal to be independent Gaussians $\cN(0, \frac{1}{d})$. The vertex $s$ is connected to the vertex $i\in (s, r]$ if $\by_i(s)\geq -b_i/\sqrt{d}$, where 
{\smaller[0.3]
\[b_i=\frac{c_p+\sqrt{d}\langle \pi_{s-1}(\by_i), \pi_{s-1}(\by_s)\rangle}{\by_s(s)}.\]
}
Note that $b_i$ is fully determined by the columns $M_1, \dots, M_{s-1}$ and the diagonal entries. \begin{claim}\label{claim:b_i inequality}
If $B^{\blue}_r$ holds, then
{\smaller[0.3]
\[\big|b_i-c_p-\sqrt{d}\langle \pi_{s-1}(\by_i), \pi_{s-1}(\by_s)\rangle\big|\leq d^{-\frac{1}{5}}.\]
}
\end{claim}
\begin{proof}
By Claim~\ref{claim:diagonal_bound}, we know that $\by_s(s)\in (1-2\delta, 1+\delta)$, and thus $\frac{1}{\by_s(s)}\in (1-2\delta, 1+3\delta)$. So, we can write
{\smaller[0.3]
\begin{align*}
\big|b_i-c_p-\sqrt{d}\langle \pi_{s-1}(\by_i), \pi_{s-1}(\by_s)\rangle\big|&=\big|\big(\frac{1}{\by_s(s)}-1\big)\big(c_p+\sqrt{d}\langle \pi_{s-1}(\by_i), \pi_{s-1}(\by_s)\rangle\big)\big|\\
&\leq 3\delta\big|c_p+\sqrt{d}\langle \pi_{s-1}(\by_i), \pi_{s-1}(\by_s)\rangle\big|.
\end{align*}
}
Also, note that $c_p$ and $\sqrt{d}\langle \pi_{s-1}(\by_i), \pi_{s-1}(\by_s)\rangle$ can be bounded by functions independent of the growing parameter $\ell$, i.e. they can be bounded by $O_{p, C}(1)$. Indeed, the definition of $c_p$ does not depend on $\ell$, and $\langle \pi_{s-1}(\by_i), \pi_{s-1}(\by_s)\rangle\leq \|\pi_{s-1}(\by_i)\|_2\cdot  \|\pi_{s-1}(\by_s)\|_2\leq \frac{\alpha_{\blue}^2\ell}{d}\leq O_{p, C}(\frac{1}{\sqrt{d}})$, due to the definition of the event $B^{\blue}_r$. Finally, since $\delta=\alpha_{\blue} d^{-1/4}=O_{p, C}(d^{-1/4})$, we get that for sufficiently large $\ell$ and $d$
{\smaller[0.3]
\[
\big|b_i-c_p-\sqrt{d}\langle \pi_{s-1}(\by_i), \pi_{s-1}(\by_s)\rangle\big|\leq 3\delta\big|c_p+\sqrt{d}\langle \pi_{s-1}(\by_i), \pi_{s-1}(\by_s)\rangle\big|\leq d^{-\frac{1}{5}}.\qedhere
\]
}
\end{proof}

Since the event $E_{s,i}$ only depends on the random variable $\by_i(s)$, connections $si$ appear independently for all $i\in (s, r]$. Using the cumulative distribution function $\Phi$, the fact that $\by_i(s)\sim \cN(0, \frac{1}{d})$, and that $E_{si}$ holds when $\by_i(s)\geq -b_i/\sqrt{d}$, we can then write
{\smaller[0.3]
\begin{equation}
  \Pb\bigg[\bigwedge_{s<i\leq r}E_{s, i}\bigg|M[s\!-\!1]\bigg]=\prod_{i=s+1}^r \Pb\big[E_{s, i}\big|M[s\!-\!1]\big]=\prod_{i=s+1}^r\Phi(b_i).  
\end{equation}
}

For each outcome of $M_s$, the induction hypothesis provides a bound on the probability of $C_s\wedge B^{\blue}_r$ over the random choice of $M_{s+1}, \dots, M_r$. Thus, in order to perform an induction step, we will integrate this bound over all outcomes of $M_s$ satisfying the event $C_s\wedge B^{\blue}_r$. At this point, it is important to observe that $C_{s-1}=C_{s}\wedge \bigwedge_{s<i\leq r}E_{si}$, i.e. that vertices $s, \dots, r$ form a clique precisely when vertices $s+1, \dots, r$ form a clique and $s$ is connected to all vertices among $s+1, \dots, r$. Hence, by the law of total probability, we can write
{\smaller[0.3]
\begin{align}\label{eq:total_probability_blue_appendix}
\Pb\Big[C_{s\!-\!1}\wedge B^{\blue}_r&\Big| M[s\!-\!1]\Big] =\bE_{M_s}\bigg[\Pb\Big[C_{s}\wedge B^{\blue}_r\Big| M_{s}\Big]\bigg| \bigwedge_{s<i\leq r}E_{s, i}, M[s\!-\!1]\bigg]\Pb\bigg[\bigwedge_{s<i\leq r}E_{s, i}\bigg|M[s\!-\!1]\bigg].
\end{align}
}

Since we have already estimated the latter term, we focus on the first one. We use the induction hypothesis to upper bound $\Pb\Big[C_{s}\wedge B^{\blue}_r\Big| M_{s}\Big]$ as follows
{\smaller[0.3]
\begin{equation}\label{eq:induction_hypothesis_blue_appendix}
    \Pb\Big[C_{s}\wedge B^{\blue}_r\Big| M_{s}\Big]\leq (1\!-\!p)^{\binom{r-s}{2}}\exp\bigg(\mu\sqrt{d}\!\!\sum_{s<i< j\leq r}\!\!\!\!\langle \pi_s(\by_i), \pi_s(\by_j)\rangle + \Big(\frac{a^3r}{(1-p)^3\sqrt d}+d^{-\frac{1}{5}}\Big)\binom{r\!-\!s}{2}\!+r-s\bigg).
\end{equation}
}

Note that most terms on the right-hand side are independent of the entries of $M_s$. In fact, the only terms which depend on it are $\exp\Big(\mu\sqrt{d}\sum_{s<i<j\leq r} \by_i(s)\by_j(s)\Big)$. Thus, we focus on bounding their expectation. The following claim corresponds to Claim~\ref{claim:auxiliary_2} in the proof of Proposition~\ref{prop:induction_1}. However, we cannot bound the left hand side by $\frac{a^3}{(1-p)^3\sqrt{d}}\binom{r-s}{2}+1$. The reason for this is that the thresholds $b_i$ may be relatively far from $c_p$, meaning that $\frac{\phi(b_i)^2}{\Phi(b_i)^2}$ could be somewhat far from $\frac{a^2}{(1-p)^2}$. This does not affect the further flow of the proof, due to the concavity of the function $f_r$, as we will shortly see in Claim~\ref{claim:auxiliary_cdf_appendix}.

\begin{claim}\label{claim:auxiliary_gaussian_appendix_blue}
If the entries of $M_s$ are independently sampled from $\cN(0, \frac{1}{d})$, and $S=\sum_{s<i<j\leq r}\by_i(s)\by_j(s)$, then we have
{\smaller[0.3]
\begin{align}
    \bE_{M_s}\bigg[\exp\bigg(\mu\sqrt{d}S\bigg)\bigg| \bigwedge_{s<i\leq r}\!\! E_{s, i}, M[s\!-\!1]\bigg]\leq  \exp\bigg(\frac{ar}{(1-p)\sqrt{d}}\sum_{i=s+1}^{r}\frac{\phi(b_i)^2}{\Phi(b_i)^2}+1\bigg).
\end{align}
}
\end{claim}
\begin{proof}
Recall that $E_{s, i}$ holds whenever $\by_i(s)\geq -b_i/\sqrt{d}$. Since the diagonal entries are fixed, conditioning on the intersection of $E_{s,i}$ leaves the variables $\by_{s+1}(s), \dots, \by_{r}(s)$ independent. Each $\by_i(s)$ follows the distribution of a lower truncated Gaussian with cutoff $-b_i/\sqrt{d}$.
Truncated Gaussian random variables have variance proxy at most $1/d$. We apply Lemma~\ref{lemma:exponential_moment_special_case} with $\lambda=\mu\sqrt{d}$, $k=r-s$, and $S=\sum_{s<i<j\leq r}\by_i(s)\by_j(s)$ to obtain
{\smaller[0.3]
\[\bE[e^{\lambda S}]\leq \exp\bigg( \lambda \bE[S] \!+\! \frac{\lambda^2k^2}{d}\sum_{j=s+1}^r (\bE \by_j(s))^2 \!+\!\frac{4\lambda k}{d}\bigg).\]
}
We must verify that $d\geq 4\lambda k$. This reduces to $d \geq 4 \mu\sqrt{d}(r-s)$. Since $r-s \leq C\ell$ and $\sqrt{d}=D\ell$, the condition becomes $D \geq 4\mu C$. Observe that it is sufficient to verify that $\mu\leq a/(1-p)$, since we already assumed that $D \geq 4aC/(1-p)\geq 4\mu C$. This is a consequence of a simple calculation as 
{\smaller[0.3]
\begin{align*}
\mu&=2\frac{ar}{(1-p)\sqrt{d}}\frac{\phi(c_p)}{\Phi(c_p)}\frac{\phi'(c_p)\Phi(c_p)-\phi(c_p)\Phi'(c_p)}{\Phi(c_p)^2}+\frac{\Phi'(c_p)}{\Phi(c_p)}\\
&=\frac{\phi(c_p)}{\Phi(c_p)}\Big(1-2\frac{ra}{(1-p)\sqrt{d}}\frac{c_p\phi(c_p)\Phi(c_p)+\phi(c_p)^2}{\Phi(c_p)^2}\Big)\leq \frac{a}{1-p}.
\end{align*}
}
The last term in the exponent is bounded by $1$, as $4\lambda k / d \leq 1$.

Using the independence of $\by_{s+1}(s), \dots, \by_r(s)$, we upper bound $\bE[S]$:
{\smaller[0.3]
\[\bE[S]=\sum_{s<i<j\leq r}\bE[\by_i(s)]\bE[\by_j(s)]\leq \frac{1}{2}\Big(\sum_{s<i\leq r}\bE[\by_i(s)]\Big)^2\leq \frac{k}{2}\sum_{s<i\leq r}\bE[\by_i(s)]^2,\]
}
where we used the inequality between the arithmetic and quadratic mean. Thus, the main terms in the exponent are bounded by:
{\smaller[0.3]
\[\lambda \bE[S] \!+\! \frac{\lambda^2k^2}{d}\sum_{j=s+1}^r (\bE \by_j(s))^2\leq \lambda \frac{k}{2}\Big(1\!+\!\frac{2\lambda k}{d}\Big)\sum_{i=s+1}^r \bE[\by_i(s)]^2.\]
}
By Lemma~\ref{lemma:truncated_expectation}, we have $\bE[\by_i(s)]^2 = \frac{\phi(b_i)^2}{d\Phi(b_i)^2}$. Substituting $\lambda = \mu\sqrt{d}$:
{\smaller[0.3]
\[ \lambda \frac{k}{2}\Big(1\!+\!\frac{2\lambda k}{d}\Big) \frac{1}{d} = \frac{\mu k}{2\sqrt{d}} \left( 1 + \frac{2\mu k}{\sqrt{d}} \right). \]
}
Since $k \le r\le C\ell$ and $D \ge 4\mu C$, we have $\frac{2\mu k}{\sqrt{d}} \le \frac{2\mu r}{D\ell}\leq \frac{1}{2}$. Thus:
{\smaller[0.3]
\[ \frac{\mu k}{2\sqrt{d}} \left( 1 + \frac{2\mu k}{\sqrt{d}} \right)\leq \frac{3}{2}\frac{\mu k}{2\sqrt{d}}\leq \frac{ar}{(1-p)\sqrt{d}}. \]
}
This yields the desired coefficient of $\frac{ar}{(1-p)\sqrt d}$, completing the proof.
\end{proof}

\begin{claim}\label{claim:auxiliary_cdf_appendix}
We have
{\smaller[0.3]
\begin{align*}
\exp\bigg(\frac{ar}{(1-p)\sqrt d}\sum_{i=s\!+\!1}^{r}\frac{\phi(b_i)^2}{\Phi(b_i)^2}&\bigg)\prod_{i=s\!+\!1}^r \Phi(b_i)\leq \\
&\leq (1\!-\!p)^{r\!-\!s} \exp\bigg(\Big(\frac{a^3r}{(1-p)^3\sqrt d}+\!d^{-\frac{1}{5}}\Big)(r\!-\!s)\!+\!\mu\sqrt{d}\sum_{i=s\!+\!1}^r \langle \pi_{s-1}(\by_i), \pi_{s-1}(\by_s)\rangle \bigg).
\end{align*}
}
\end{claim}
\begin{proof}
Recall the function $f_r(x)=\frac{ar}{(1-p)\sqrt d}\frac{\phi(x)^2}{\Phi(x)^2}+\log \Phi(x)$.  Observe that 
{\smaller[0.3]
\[\exp\bigg(\frac{ar}{(1-p)\sqrt d}\sum_{i=s+1}^{r}\frac{\phi(b_i)^2}{\Phi(b_i)^2}\bigg)\prod_{i=s+1}^r \Phi(b_i)=\prod_{i=s+1}^r e^{f_r(b_i)}.\]
}
By Lemma~\ref{lemma:function_analysis}, applied with $\kappa=\frac{(1-p)\sqrt d}{ar}\geq 4$, the function $f_r$ is concave, so $f_r(x)\leq f_r(c_p)+f_r'(c_p)(x-c_p)$. We have $f_r(c_p) = \frac{ar}{(1-p)\sqrt d}\frac{\phi(c_p)^2}{\Phi(c_p)^2} + \log\Phi(c_p) = \frac{a^3r}{(1-p)^3\sqrt d} + \log(1-p)$. Also, $f_r'(c_p)=\mu$.
{\smaller[0.3]
\begin{align*}
\sum_{i=s+1}^r f_r(b_i) &\leq \sum_{i=s+1}^r \left( f_r(c_p) + \mu(b_i - c_p) \right) = (r-s)\left(\log(1-p) + \frac{a^3r}{(1-p)^3\sqrt d}\right) + \mu \sum_{i=s+1}^r (b_i - c_p).
\end{align*}
}
Using the bound $|b_i - c_p - \sqrt{d}\langle \pi_{s-1}(\by_i), \pi_{s-1}(\by_s)\rangle| \leq d^{-1/5}$, we substitute for $b_i - c_p$:
{\smaller[0.3]
\[ \mu \sum_{i=s+1}^r (b_i - c_p) \leq \mu \sum_{i=s+1}^r \left( \sqrt{d}\langle \pi_{s-1}(\by_i), \pi_{s-1}(\by_s)\rangle + d^{-1/5} \right). \]
}
Combining these terms and using that $\mu\leq 1$ gives the result.
\end{proof}

We now combine the claims to finish the proof of Proposition~\ref{prop:induction_appendix_blue}. Plugging the result of Claim~\ref{claim:auxiliary_gaussian_appendix_blue} into the expectation in (\ref{eq:total_probability_blue}), and then applying Claim~\ref{claim:auxiliary_cdf_appendix}, we obtain:
{\smaller[0.3]
\begin{align*}
\Pb\Big[C_{s\!-\!1}\wedge B^{\blue}_r\Big| M[s\!-\!1]\Big] &\leq (1\!-\!p)^{r-s} \exp\bigg(  \mu\sqrt{d} \sum_{i=s\!+\!1}^r \langle \pi_{s-1}(\by_i), \pi_{s-1}(\by_s) \rangle \!+\!\Big(\frac{a^3r}{(1-p)^3\sqrt d}\!+\!d^{-\frac{1}{5}}\Big)(r\!-\!s) \!+\!1\bigg)\\
\cdot (1\!-\!p)^{\binom{r-s}{2}} &\exp\bigg( \mu\sqrt{d} \sum_{s<i<j\leq r} \langle \pi_{s-1}(\by_i), \pi_{s-1}(\by_j) \rangle  \!+\! \Big(\frac{a^3r}{(1-p)^3\sqrt d}\!+\!d^{-1/5}\Big)\binom{r\!-\!s}{2} \!+\! (r\!-\!s) \bigg).
\end{align*}
}
The inner product terms merge to form the sum over $s-1 < i < j \le r$. The leading terms combine as:
{\smaller[0.3]
\[ \left(\frac{a^3r}{(1-p)^3\sqrt d} + d^{-1/5}\right) \binom{r-s}{2} + \left(\frac{a^3r}{(1-p)^3\sqrt d} + d^{-1/5}\right) (r-s) = \left(\frac{a^3r}{(1-p)^3\sqrt d} + d^{-1/5}\right) \binom{r-s+1}{2}. \]
}
Finally, the linear terms are simply $(r-s) + 1$, which is the required linear error term for the step $s-1$. This completes the induction step.
\end{proof}

\subsection{Red cliques.} Let us now give an upper bound on $P_{\red, r}^*$, which corresponds to Proposition~\ref{prop:perfect_probabilities_appendix_blue}.

\begin{proposition}\label{prop:perfect_probabilities_appendix_red}
Let $C>1, p\in (0, 1)$ be real numbers and $\ell, d=D^2\ell^2$ be sufficiently large positive integers, as a function of $C$ and $p$. Additionally, assume that $D/\sqrt{\log p^{-1}}$ is a sufficiently large absolute constant. Then, for every integer $3\leq r\leq \ell$, we have
{\smaller[0.3]
\begin{align}
    P_{\red, r}^*\leq p^{\binom{r}{2}}\exp\bigg(-\Big(1+O\Big(\frac{\sqrt{\log p^{-1}}}{D}\Big)\Big)\frac{a^3}{p^3\sqrt{d}}\binom{r}{3}+d^{-1/5}\binom{r}{2}+r\bigg),\label{eq:perfect_probabilities_red_appendix}
\end{align}
}
where $a=\phi(c_p)$.
\end{proposition}

Note that the linear and quadratics terms from in the statement are immaterial - they are only present for bookkeeping purposes and can be absorbed into the error term $O\big(\frac{\sqrt{\log p^{-1}}}{D}\binom{r}{3}\big)$ when $r\approx \sqrt{d}$. As before, we derive this proposition by induction, with the general induction statement given in the following lemma. Recall, $I_s$ denotes the event that the vertices $s+1, \dots, r$ form a red clique (or equivalently, an independent set in the context of Definition~\ref{def:RGG}). 

\begin{proposition}\label{prop:induction_appendix_red}
Under the assumptions of Proposition~\ref{prop:perfect_probabilities_appendix_red}, let $s, r$ be nonnegative integers such that $1\leq s+1\leq r\leq \ell$. Suppose that the first $s$ columns of $M$ are fixed, and denote them by $M[s]$. Then
{\smaller[0.3]
\begin{align}
    \mathbb{P}\Big[I_s\wedge B^{\red}_r\Big| M[s]\Big]\!\leq p^{\binom{r-s}{2}}\exp\bigg(-\frac{a\sqrt{d}}{p}\sum_{s<i< j\leq r}\langle \pi_s(\by_i), \pi_s(\by_j)\rangle -\Big(1+O\big(& \frac{\sqrt{\log p^{-1}}}{D}\big)\Big)\frac{a^3}{p^3\sqrt{d}}\binom{r-s}{3} \notag\\
    &+\frac{a}{p}d^{-1/5}\binom{r-s}{2} +(r-s)\bigg). \label{eq:induction_bound_red_2_appendix}
\end{align}
}
\end{proposition}
\begin{proof}
We argue by reverse induction on $s$. The base case $s=r$ is trivial. Thus, let us show how to do a step of the induction, from $s$ to $s-1$. We are given the columns $M_1, \dots, M_{s-1}$, and we sample the column $M_s$ by setting all of its entries below the diagonal to be independent Gaussians $\cN(0, \frac{1}{d})$. The vertex $s$ forms a red edge to the vertex $i\in (s, r]$ if $\by_i(s)\leq -b_i/\sqrt{d}$, where 
{\smaller[0.3]
\[b_i=\frac{c_p+\sqrt{d}\langle \pi_{s-1}(\by_i), \pi_{s-1}(\by_s)\rangle}{\by_s(s)}.\]
}
Note that $b_i$ is fully determined by the columns $M_1, \dots, M_{s-1}$ and the diagonal entries. Also, by essentially the same argument as in Claim~\ref{claim:b_i inequality}, we have that if $B^{\red}_r$ holds then
{\smaller[0.3]
\[\big|b_i-c_p-\sqrt{d}\langle \pi_{s-1}(\by_i), \pi_{s-1}(\by_s)\rangle\big|\leq d^{-\frac{1}{5}}.\]
}

Since the event $\overline{E}_{s,i}$ only depends on the random variable $\by_i(s)$, the colors of the edges $si$ are independent for all $i\in (s, r]$. Using the cumulative distribution function $\Phi$, the fact that $\by_i(s)\sim \cN(0, \frac{1}{d})$, and that $E_{si}$ holds when $\by_i(s)\geq -b_i/\sqrt{d}$, we can then write
{\smaller[0.3]
\begin{align}
  \Pb\bigg[\bigwedge_{s<i\leq r}\overline{E}_{s, i}\bigg|M[s\!-\!1]\bigg]&=\prod_{i=s+1}^r \Pb\big[\,\overline{E}_{s, i}\big|M[s\!-\!1]\big]
  =\prod_{i=s+1}^r\Phi(-b_i)
  \leq p^{r-s}\prod_{i=s+1}^r e^{-\frac{a}{p}(b_i-c_p)}\\
  &\leq p^{r-s}\exp\Big(-\frac{a}{p}\sum_{i=s+1}^r\sqrt d\langle\pi_{s-1}(\by_i),\pi_{s-1}(\by_s)\rangle+\frac{a}{p}(r-s)d^{-1/5}\Big).\notag  
\end{align}
}
In the last inequality, we have used that $\Phi(x)$ is log-concave (see (\ref{eqn:log_concavity})).

For each outcome of $M_s$, the induction hypothesis provides a bound on the probability of $I_s\wedge B^{\red}_r$ over the random choice of $M_{s+1}, \dots, M_r$. Thus, in order to perform an induction step, we will integrate this bound over all outcomes of $M_s$ satisfying the event $I_s\wedge B^{\red}_r$. At this point, it is important to observe that $I_{s-1}=I_{s}\wedge \bigwedge_{s<i\leq r}\overline{E}_{si}$, i.e. that vertices $s, \dots, r$ form a red clique precisely when vertices $s+1, \dots, r$ form a red clique and $s$ has a red edge to all vertices among $s+1, \dots, r$. Hence, by the law of total probability, we can write
{\smaller[0.3]
\begin{align}\label{eq:total_probability_red}
\Pb\Big[I_{s -\!1}\wedge B^{\red}_r&\Big| M[s\!-\!1]\Big] =\bE_{M_s}\bigg[\Pb\Big[I_{s}\wedge B^{\red}_r\Big| M_{s}\Big]\bigg| \bigwedge_{s<i\leq r}\overline{E}_{s, i}, M[s\!-\!1]\bigg]\Pb\bigg[\bigwedge_{s<i\leq r}\overline{E}_{s, i}\bigg|M[s\!-\!1]\bigg].
\end{align}
}

Since we have already estimated the latter term, we focus on the first one. We use the induction hypothesis to upper bound $\Pb\Big[I_{s}\wedge B^{\red}_r\Big| M_{s}\Big]$ as follows
{\smaller[0.3]
\begin{align}\label{eq:induction_hypothesis_red}
    \Pb\Big[I_{s}\wedge B^{\red}_r\Big| M_{s}\Big]\leq p^{\binom{r-s}{2}}\exp\bigg(- \frac{a\sqrt{d}}{p}\!\!\sum_{s<i< j\leq r}\!\!\!\!\langle \pi_s(\by_i), \pi_s(\by_j)\rangle - \Big(1+&O\big( \frac{\sqrt{\log p^{-1}}}{D}\big)\Big)\frac{a^3}{p^3\sqrt{d}}\binom{r\!-\!s}{3}\! \\
    &+\frac{a}{p}d^{-1/5}\binom{r-s}{2}+(r-s)\bigg).\notag
\end{align}
}
Note that most terms on the right-hand side are independent of the entries of $M_s$. In fact, the only terms which depend on it are $\exp\Big(- \frac{a\sqrt{d}}{p}\sum_{s<i<j\leq r} \by_i(s)\by_j(s)\Big)$. Thus, we focus on bounding their expectation.

\begin{claim}\label{claim:auxiliary_gaussian_appendix_red}
If the entries of $M_s$ are independently sampled from $\cN(0, \frac{1}{d})$, and $S=\sum_{s<i<j\leq r}\by_i(s)\by_j(s)$ then we have
{\smaller[0.3]
\begin{align}
    \bE_{M_s}\bigg[\exp\bigg(- \frac{a\sqrt{d}}{p}S\bigg)\bigg| \bigwedge_{s<i\leq r}\!\! \overline{E}_{s, i}, M[s\!-\!1]\bigg]\leq  \exp\bigg(-\Big(1+O\big( \frac{\sqrt{\log p^{-1}}}{D}\big)\Big)\frac{a^3}{p^3\sqrt{d}}\binom{r-s}{2}+1\bigg).
\end{align}
}
\end{claim}
\begin{proof}
Recall that $\overline{E}_{s, i}$ holds whenever $\by_i(s)\leq -b_i/\sqrt{d}$. Since the diagonal entries are fixed, conditioning on the intersection of $\overline{E}_{s, i}$ leaves the variables $\by_{s+1}(s), \dots, \by_{r}(s)$ independent. Each $\by_i(s)$ follows the distribution of an upper truncated Gaussian with cutoff $-b_i/\sqrt{d}$.
Truncated Gaussian random variables have variance proxy at most $1/d$. We apply Lemma~\ref{lemma:exponential_moment_special_case} with $\lambda=-\frac{a\sqrt{d}}{p}$, $k=r-s$, and $S=\sum_{s<i<j\leq r}\by_i(s)\by_j(s)$ to obtain
{\smaller[0.3]
\[\bE[e^{\lambda S}]\leq \exp\bigg( \lambda \bE[S] \!+\! \frac{\lambda^2k^2}{d}\sum_{j=s+1}^r \bE [\by_j(s)]^2 \!+\!\frac{4|\lambda| k}{d}\bigg).\]
}
We must verify that $d\geq 4|\lambda| k$. This reduces to $d \geq 4  \frac{a\sqrt{d}}{p}(r-s)$. Since $r-s \leq \ell$ and $\sqrt{d}=D\ell$, the condition becomes $D \geq 4 \frac{a}{p}$. This is true, since we have assumed that $D/\sqrt{\log p^{-1}}$ is at least a large absolute constant and $\frac{a}{p}=(1+o(1))\sqrt{2\log p^{-1}}$ as $p\to 0$. So, the last term in the exponent is bounded by $1$, as $4|\lambda| k / d \leq 1$.

Next, to evaluate $\bE [\by_j(s)]$, we apply Lemma~\ref{lemma:truncated_expectation} to write
{\smaller[0.3]
\[\bE [\by_j(s)]=\frac{-\phi(c_p)}{\sqrt{d}\Phi(-c_p)}+O(\frac{|c_p-b_i|}{\sqrt{d}})=\frac{-a}{p\sqrt{d}}+O\Big(|\langle \pi_{s-1}(\by_j), \pi_{s-1}(\by_s)\rangle|+d^{-1/5-1/2}\Big).\]
}
Since $|\langle \pi_{s-1}(\by_j), \pi_{s-1}(\by_s)\rangle|\leq \|\pi_{s-1}(\by_j)\|\cdot \|\pi_{s-1}(\by_s)\|\leq \frac{\alpha_{\red}^2\ell}{d}\leq O(\frac{\log p^{-1}}{D\sqrt{d}})$ by the Cauchy-Schwarz inequality, we have $\bE [\by_j(s)]=-\frac{a}{p\sqrt{d}}+O\big(\frac{\log p^{-1}}{D\sqrt{d}}\big)$, so $\bE [\by_j(s)]^2\leq \frac{2a^2}{p^2d}$. 

Thus, the middle term in the exponent is bounded by
{\smaller[0.3]
\begin{align*}
\frac{\lambda^2k^2}{d}\!\!\!\sum_{j=s+1}^r \bE [\by_j(s)]^2\!\leq \frac{a^2k^2}{p^2}\cdot \ell\frac{2a^2}{p^2d}\leq \frac{a^4}{p^4D\sqrt{d}}k^2\!=\!O\Big( \frac{\sqrt{\log p^{-1}}}{D}\Big)\frac{a^3}{p^3\sqrt{d}}\binom{k}{2}\!=\!O\Big( \frac{\sqrt{\log p^{-1}}}{D}\Big)\frac{a^3}{p^3\sqrt{d}}\binom{r\!-\!s}{2},
\end{align*}
}
where we used that $a/p=\Theta((\log p^{-1})^{1/2})$ and $r-s\leq \ell$. Finally, by Lemma~\ref{lemma:truncated_expectation}
{\smaller[0.3]
\begin{align*}
\lambda \bE[S]&=-\frac{a\sqrt{d}}{p}\sum_{s<i<j\leq r}\frac{\phi(b_i)}{\sqrt{d}\Phi(-b_i)}\frac{\phi(b_j)}{\sqrt{d}\Phi(-b_j)}\\
&=-\frac{a}{p\sqrt{d}}\sum_{s<i<j\leq r}\Big(\frac{a}{p}+O\big(\frac{\log p^{-1}}{D}\big)\Big)\Big(\frac{a}{p}+O\big(\frac{\log p^{-1}}{D}\big)\Big)=-\frac{a}{p\sqrt{d}}\sum_{s<i<j\leq r}\bigg(\frac{a^2}{p^2}+O\Big(\frac{a(\log p^{-1})}{pD}\Big)\bigg)\\
    &=-\frac{a^3}{p^3\sqrt{d}}\binom{r-s}{2}+O\Big(\frac{a^2(\log p^{-1})}{p^2D\sqrt{d}}\binom{r-s}{2}\Big)=-\Big(1+O\big( \frac{\sqrt{\log p^{-1}}}{D}\big)\Big)\frac{a^3}{p^3\sqrt{d}}\binom{r-s}{2}
\end{align*}
}
In the last line, we have used that $a/p=\Theta((\log p^{-1})^{1/2})$.
\end{proof}

We now combine the claims to finish the proof of Proposition~\ref{prop:induction_appendix_red}. Plugging the result of Claim~\ref{claim:auxiliary_gaussian_appendix_red} into the expectation in (\ref{eq:total_probability_red}), we obtain:
{\smaller[0.3]
\begin{align*}
\Pb\Big[I_{s-1}\wedge B^{\red}_r\Big| M[s-1]\Big] \leq p^{\binom{r-s}{2}} \exp\bigg( &- \frac{a\sqrt{d}}{p} \sum_{s<i<j\leq r} \langle \pi_{s-1}(\by_i), \pi_{s-1}(\by_j) \rangle  \\
- \Big(1+&O\big( \frac{\sqrt{\log p^{-1}}}{D}\big)\Big) \frac{a^3}{p^3\sqrt{d}}\binom{r-s}{3} +\frac{a}{p}d^{-1/5}\binom{r-s}{2}+(r-s) \bigg) \\
 \cdot p^{r-s} \exp\bigg(\!\! -\!\! \Big(1\!+\!O\big(\! \frac{\sqrt{\log p^{-1}}}{D}\big)\!\Big)\frac{a^3}{p^3\sqrt{d}}&\binom{r\!-\!s}{2} \! -\!\frac{a\sqrt{d}}{p} \!\!\!\sum_{i=s+1}^r\!\! \langle \pi_{s-1}(\by_i), \pi_{s-1}(\by_s) \rangle \!+\!\frac{a}{p}d^{-1/5}(r\!-\!s)\!+\!1\!\bigg).
\end{align*}
}
The inner product terms merge to form the sum over $s-1<  i < j \le r$. The main terms combine as:
{\smaller[0.3]
\[ \Big(1+O\big( \frac{\sqrt{\log p^{-1}}}{D}\big)\Big)\frac{a^3}{p^3\sqrt{d}}\binom{r\!-\!s}{2}+ \Big(1+O\big( \frac{\sqrt{\log p^{-1}}}{D}\big)\Big) \frac{a^3}{p^3\sqrt{d}}\binom{r\!-\!s}{3}\!=\!\Big(1+O\big( \frac{\sqrt{\log p^{-1}}}{D}\big)\Big) \frac{a^3}{p^3\sqrt{d}}\binom{r\!-\!s\!+\!1}{3}\]
}
Finally, the linear and quadratic terms also combine to give the required error terms for the step $s-1$. This completes the induction step.
\end{proof}

\subsection{Bookkeeping to obtain the final bounds.}

In this section, the goal is to derive Theorem~\ref{thm:main_appendix} from Propositions~\ref{prop:perfect_probabilities_appendix_blue} and~\ref{prop:perfect_probabilities_appendix_red}. This will be done in a very similar fashion as we derived Theorem~\ref{thm:main} in the introduction, just with slightly different constants.

Before we start, we will select an appropriate value of $D$. To ensure that the assumptions of the Proposition~\ref{prop:perfect_probabilities_appendix_blue} are satisfies, we set $D= 4aC/(1-p)$. It is worth noting that, once we set $p=p_C+\frac{1}{C}$, this gives $D=(4\sqrt{2}+o(1))(\log p^{-1})^{3/2}$ as $C\to\infty$, since $C=(1+o(1))p_C^{-1} \log p_C^{-1}$ and $a=(1+o(1))p\sqrt{\log p^{-1}}$.

\begin{lemma}\label{lemma:bookkeeping_appendix}
Let $C>1$ be a sufficiently large constant and let $p_C\in (0, 1/2)$ be the unique solution to $C=\frac{\log p_C}{\log (1-p_C)}$. Then, if $p=p_C+\frac{1}{C}$ and $D=4\frac{aC}{1-p}$, then for every sufficiently large $\ell$ the clique probabilities in $G(n, d, p)$  with $d=D^2\ell^2$ are bounded by
{\smaller[0.3]
\[P_{{\rm red}, \ell}\leq \Big(e^{-1/10}p_C\Big)^{\binom{\ell}{2}} \text{ and } P_{{\rm blue}, C\ell}\leq \Big(e^{-1/3C}(1-p_C)\Big)^{\binom{C\ell}{2}}.\]
}
\end{lemma}
\begin{proof}[Proof of Lemma~\ref{lemma:bookkeeping_appendix}.]
In this appendix, we have focused on bounding the probabilities $P_{\red, r}^*, P_{\blue, r}^*$, but the same bounds apply for their non-perfect counterparts. Indeed, the argument from Section~\ref{sec:perfect} only uses that the probability of violating the relevant perfectness condition is negligible compared to the target clique probability; the choices of $\alpha_{\red}$ and $\alpha_{\blue}$ were made precisely to ensure this for red cliques of order at most $\ell$ and blue cliques of order at most $C\ell$. By using that $\frac{\ell}{3}\binom{\ell}{2}\geq \binom{\ell}{3}$ and that $\sqrt{d}=D\ell$, we can transform the corresponding bounds to
{\smaller[0.3]
\begin{align*}
P_{\red, \ell}&\leq p^{\binom{\ell}{2}}\exp\Big(-\Big(1+O\Big(\frac{\sqrt{\log p^{-1}}}{D}\Big)\Big)\frac{a^3}{p^3\sqrt{d}} \binom{\ell}{3}\Big)\leq p^{\binom{\ell}{2}}\exp\Big(-(1+o(1))\frac{a^3}{3p^3D}\binom{\ell}{2}\Big),\\
P_{\blue, C\ell}&\leq (1-p)^{\binom{C\ell}{2}}\exp\Big((1+o(1))\frac{4a^3}{(1-p)^3\sqrt{d}}\binom{C\ell}{3}\Big)\leq (1-p)^{\binom{C\ell}{2}}\exp\Big(\frac{(4+o(1))a^3C}{3(1-p)^3D}\binom{C\ell}{2}\Big),
\end{align*}
}
where $o(1)\to 0$ as $C\to \infty$ and then $\ell\to\infty$.

Let us now plug in $p=p_C+\frac{1}{C}$, $D=4aC/(1-p)$ to obtain
{\smaller[0.3]
\begin{align*}
P_{\red, \ell}&\leq p^{\binom{\ell}{2}}\exp\Big(-(1+o(1))\frac{a^3}{3p^3D}\binom{\ell}{2}\Big)\leq \Big((p_C+1/C) \exp\Big(-(1+o(1))\frac{(2\log p^{-1})^{3/2}}{3\cdot 4\sqrt{2}(\log p^{-1})^{3/2}}\Big)\Big)^{\binom{\ell}{2}}\\
&=\Big((p_C+1/C) e^{-1/6+o(1)}\Big)^{\binom{\ell}{2}}\leq \Big(p_Ce^{-1/10}\Big)^{\binom{\ell}{2}},
\end{align*}
}
where we have used that $1/C\ll p_C$ when $C\to \infty$.

On the other hand, we have 
{\smaller[0.3]
\begin{align*}
P_{\blue, C\ell}&\leq (1-p_C-1/C)^{\binom{C\ell}{2}}\exp\Big(\frac{(4+o(1))a^3C}{3(1-p)^3D}\binom{C\ell}{2}\Big).
\end{align*}
}
We now use that $1-p_C-1/C=(1-p_C)\big(1-\frac{1}{C(1-p_C)}\big)\leq (1-p_C)\exp\big(-\frac{1}{C(1-p_C)}\big)$.
Hence,
{\smaller[0.3]
\begin{align*}
P_{\blue, C\ell}&\leq \Big((1-p_C)\exp\Big(-\frac{1}{C(1-p_C)}+\frac{(4+o(1))a^3C}{3(1-p)^3D}\Big)\Big)^{\binom{C\ell}{2}}.
\end{align*}
}
Finally, observe that $\frac{(4+o(1))a^3C}{3(1-p)^3D}\leq \frac{a^2}{3(1-p)^2}=O(p^2 \log p^{-1})=o(1/C)$, and therefore, for sufficiently large $C$ we have 
{\smaller[0.3]
\begin{align*}
P_{\blue, C\ell}&\leq \Big((1-p_C)e^{-\frac{1}{3C}}\Big)^{\binom{C\ell}{2}}.\qedhere
\end{align*}
}
\end{proof}

\begin{proof}[Proof of Theorem~\ref{thm:main_appendix}.]
By Lemma~\ref{lemma:bookkeeping_appendix}, for $p=p_C+\frac{1}{C}$ and $D=4\frac{aC}{1-p}$, the clique probabilities in $G(n, d, p)$ (where $d=D^2\ell^2$) satisfy
{\smaller[0.3]
\[P_{{\rm red}, \ell}\leq \Big(e^{-1/10}p_C\Big)^{\binom{\ell}{2}} \text{ and } P_{{\rm blue}, C\ell}\leq \Big(e^{-1/3C}(1-p_C)\Big)^{\binom{C\ell}{2}}.\]
}
Also, set $n=(e^{1/24}p_C^{-1/2})^\ell$, and let us show that the red/blue coloring of $K_n$ defined using the Gaussian random geometric graph has no red $\ell$-clique and no blue $C\ell$-clique with high probability.

By the union bound the probability that this coloring has a red $\ell$-clique is at most
{\smaller[0.3]
\[\binom{n}{\ell}\cdot P_{\red, \ell}\leq \frac{n^{\ell}}{\ell!}(e^{-1/10}p_C)^{\frac{\ell(\ell-1)}{2}}\leq \frac{1}{\ell!}\Big(ne^{-(\ell-1)/20}p_C^{\frac{\ell-1}{2}}\Big)^\ell\leq \frac{1}{\ell!}\Big(p_C^{-\ell/2}e^{\ell/24}e^{-(\ell-1)/20}p_C^{\frac{\ell-1}{2}}\Big)^\ell.\]
}
The expression inside the parentheses is $p_C^{-1/2}\exp(\ell/24-\ell/20+1/20)$, which is at most $1$ once $\ell$ is sufficiently large as a function of $C$. Hence, the probability that there exists a red $\ell$-clique is at most $\frac{1}{\ell!}$.

Similarly, the probability that the coloring contains a blue $C\ell$-clique is at most
{\smaller[0.3]
\[\binom{n}{C\ell}\cdot P_{\blue, C\ell}\leq \frac{n^{C\ell}}{(C\ell)!}\big(e^{-1/3C}(1-p_C)\big)^{\frac{C\ell(C\ell-1)}{2}}\leq \frac{1}{(C\ell)!}\Big(ne^{-(C\ell-1)/6C}(1-p_C)^{\frac{C\ell-1}{2}}\Big)^{C\ell}.\]
}
Again, using $(1-p_C)^C=p_C$, the expression inside the parentheses is at most $\exp(\ell/24-\ell/6+O(1))$, which is at most $1$ once $\ell$ is sufficiently large as a function of $C$. Hence, the probability that there exists a blue $C\ell$-clique is at most $\frac{1}{(C\ell)!}$. This shows that there is an outcome without red $\ell$-cliques or blue $C\ell$-cliques, completing the proof.
\end{proof}


\begin{thebibliography}{100}

\bibitem{BBCGHMST}
P. Balister, B. Bollob\'as, M. Campos, S. Griffiths, E. Hurley, R. Morris, J. Sahasrabudhe and M. Tiba.
\textit{Upper bounds for multicolour Ramsey numbers.}
J. Amer. Math. Soc., 39 (2026), 765-780.

\bibitem{BMA24}
M. Barreto, O. Marchal, and J. Arbel.
\emph{Optimal sub-Gaussian variance proxy for truncated Gaussian and exponential random variables.}
Stat. Probab. Lett. 228 (2026).

\bibitem{BDER16}
S. Bubeck, J. Ding, R. Eldan and M.Z. R\'acz.
\textit{Testing for high-dimensional geometry in random graphs.}
Random Struct. Alg., 49 (2016), 503–532.

\bibitem{CGMS}
M. Campos, S. Griffiths, R. Morris and J. Sahasrabudhe.
\textit{An exponential improvement for diagonal Ramsey.}
Ann. of Math., 203 (2026) 869-932.

\bibitem{CJMS25}
M. Campos, M. Jenssen, M. Michelen and J. Sahasrabudhe.
\textit{A new lower bound for the Ramsey numbers $R(3,k)$.}
arXiv:2505.13371.

\bibitem{Con09}
D. Conlon.
\textit{A new upper bound for diagonal Ramsey numbers.}
Ann. of Math., 170 (2009), 941–960.

\bibitem{CF21}
D. Conlon and A. Ferber.
\textit{Lower bounds for multicolour Ramsey numbers.}
Adv. Math., 378 (2021), Paper No. 107528, 5 pp.

\bibitem{DGLU11}
L. Devroye, A. Gy\"orgy, G. Lugosi and F. Udina.
\textit{High-dimensional random geometric graphs and their clique number.}
Electron. J. Probab., 16 (2011), Paper no. 90, 2481–2508.

\bibitem{Erd47}
P. Erd\H{o}s.
\textit{Some remarks on the theory of graphs.}
Bull. Amer. Math. Soc., 53 (1947), 292–294.

\bibitem{ES35}
P. Erd\H{o}s and G. Szekeres.
\textit{A combinatorial problem in geometry.}
Compos. Math., 2 (1935), 463–470.

\bibitem{Gor41}
R.D. Gordon.
\textit{Values of Mills' ratio of area to bounding ordinate and of the normal probability integral for large values of the argument.}
Ann. Math. Statist., 12 (1941), 364–366.

\bibitem{GNWW24}
P. Gupta, N. Ndiaye, S. Norin and L. Wei.
\textit{Optimizing the CGMS upper bound on Ramsey numbers.}
arXiv:2407.19026.

\bibitem{HHKP25}
Z. Hefty, P. Horn, D. King and F. Pfender.
\textit{Improving $R(3,k)$ in just two bites.}
arXiv:2510.19718.

\bibitem{LM}
B. Laurent and P. Massart.
\textit{Adaptive estimation of a quadratic functional by model selection} 
Ann. Stat., 28 (2000): 1302--1338.

\bibitem{MSX25}
J. Ma, W. Shen and S. Xie.
\textit{An exponential improvement for Ramsey lower bounds.}
Invent. math. (2026). 

\bibitem{MV24}
S. Mattheus and J. Verstra\"{e}te.
\textit{The asymptotics of $r(4, t)$.}
Ann. of Math., 199 (2024), 919--941.

\bibitem{MV19}
D. Mubayi and J. Verstra\"{e}te.
\textit{A note on pseudorandom Ramsey graphs.}
J. Eur. Math. Soc., 26 (2024) 153--161.

\bibitem{Pen03}
M. Penrose.
\textit{Random Geometric Graphs.}
Oxford Studies in Probability, 5, Oxford University Press, 2003.

\bibitem{Ram30}
F.P. Ramsey.
\textit{On a Problem of Formal Logic.}
Proc. London Math. Soc., 30 (1930), 264–286.

\bibitem{Sah23}
A. Sah.
\textit{Diagonal Ramsey via effective quasirandomness.}
Duke Math. J., 172 (2023), 545–567.

\bibitem{Sampford}
M. R. Sampford.
\textit{Some inequalities on Mill's ratio and related functions.}
Ann. Math. Statist., 24 (1953), 130--132.

\bibitem{Spe75}
J. Spencer.
\textit{Ramsey's theorem---A new lower bound.}
J. Combin. Theory Ser. A, 18 (1975), 108–115.

\bibitem{Tho88}
A. Thomason.
\textit{An upper bound for some Ramsey numbers.}
J. Graph Theory, 12 (1988), 509–517.

\bibitem{Ver18}
R. Vershynin.
\textit{High-Dimensional Probability: An Introduction with Applications in Data Science.}
Cambridge Series in Statistical and Probabilistic Mathematics, 47, Cambridge University Press, 2018.

\end{thebibliography}
\end{document}